\documentclass[a4paper,10pt, oneside]{article}
\usepackage{amsmath,amssymb,amsthm,bm,mathrsfs,mathtools,graphicx}
\usepackage{algorithm,algorithmic}
\usepackage{authblk}
\usepackage[font=small,labelfont=md,textfont=it]{caption}
\usepackage[titletoc, title]{appendix}
\usepackage{enumitem}
\usepackage[colorlinks,linkcolor=black,citecolor=black]{hyperref}
\usepackage{etoolbox}
\usepackage[a4paper,paperheight=297mm,paperwidth=210mm,margin=1.25in]{geometry} 
\usepackage{longtable}
\usepackage{diagbox}
\usepackage{booktabs,makecell,multirow}
\usepackage[capitalise,nameinlink]{cleveref}
\usepackage{cases,color}

\Crefname{figure}{Figure}{Figures}
\crefformat{equation}{\textup{#2(#1)#3}}
\crefrangeformat{equation}{\textup{#3(#1)#4--#5(#2)#6}}
\crefmultiformat{equation}{\textup{#2(#1)#3}}{ and \textup{#2(#1)#3}}
{, \textup{#2(#1)#3}}{, and \textup{#2(#1)#3}}
\crefrangemultiformat{equation}{\textup{#3(#1)#4--#5(#2)#6}}%
{ and \textup{#3(#1)#4--#5(#2)#6}}{, \textup{#3(#1)#4--#5(#2)#6}}{, and \textup{#3(#1)#4--#5(#2)#6}}
\Crefformat{equation}{#2Equation~\textup{(#1)}#3}
\Crefrangeformat{equation}{Equations~\textup{#3(#1)#4--#5(#2)#6}}
\Crefmultiformat{equation}{Equations~\textup{#2(#1)#3}}{ and \textup{#2(#1)#3}}
{, \textup{#2(#1)#3}}{, and \textup{#2(#1)#3}}
\Crefrangemultiformat{equation}{Equations~\textup{#3(#1)#4--#5(#2)#6}}%
{ and \textup{#3(#1)#4--#5(#2)#6}}{, \textup{#3(#1)#4--#5(#2)#6}}{, and \textup{#3(#1)#4--#5(#2)#6}}

\apptocmd{\sloppy}{\hbadness 10000\relax}{}{}

\newcommand{\dual}[1]{\langle {#1} \rangle}

\newcommand{\norm}[1]{\lVert {#1} \rVert}
\newcommand{\normb}[1]{\big\lVert {#1} \big\rVert}

\newcommand{\abs}[1]{\lvert {#1} \rvert}
\newcommand{\ssnm}[1]
{
	\left\vert\kern-0.25ex
	\left\vert\kern-0.25ex
	\left\vert
	{#1}
	\right\vert\kern-0.25ex
	\right\vert\kern-0.25ex
	\right\vert
}

\makeatletter
\def\spher@harm#1{%
	\vbox{\hbox{%
			\offinterlineskip
			\valign{&\hb@xt@2\p@{\hss$##$\hss}\vskip.2ex\cr#1\crcr}%
		}\vskip-.36ex}%
}
\def\gshone{\spher@harm{.}}
\def\gshtwo{\spher@harm{.&.}}
\def\gshthree{\spher@harm{.&.&.}}
\let\gsh\spher@harm
\makeatother

\newtheorem{lemma}{Lemma}[section]
\newtheorem{remark}{Remark}[section]
\newtheorem{theorem}{Theorem}[section]

\numberwithin{equation}{section}
\numberwithin{figure}{section}

\makeatletter\def\@captype{table}\makeatother

\begin{document}

\title{
  \Large\bf Discrete stochastic maximal $ L^p $-regularity and convergence of a spatial
  semidiscretization for a linear stochastic heat equation\thanks{This work was partially supported by National Natural
    Science Foundation of China under grant 12301525 and
    Natural Science Foundation of Sichuan province under grant 2023NSFSC1324.
  }
}


\author[1]{Binjie Li\thanks{libinjie@scu.edu.cn}}
\author[2]{Qin Zhou\thanks{zqmath@cwnu.edu.cn}}
\affil[1]{School of Mathematics, Sichuan University, Chengdu 610064, China}
\affil[2]{School of Mathematics, China West Normal University, Nanchong 637002, China}


\date{}
\maketitle

\begin{abstract}
This study investigates the boundedness of the \( H^\infty \)-calculus for the negative discrete Laplace operator under homogeneous Dirichlet boundary conditions. The discrete operator is implemented using the finite element method, and we establish that the boundedness constant of its \( H^\infty \)-calculus remains uniformly bounded with respect to the spatial mesh size. Based on this result, we derive a discrete stochastic maximal \( L^p \)-regularity estimate for a spatial semidiscretization of a linear stochastic heat equation. Furthermore, within the framework of general spatial \( L^q \)-norms, we provide a nearly optimal pathwise uniform convergence estimate for this semidiscretization.
\end{abstract}

\medskip\noindent{\bf Keywords:}
$H^\infty$-calculus, discrete stochastic maximal $ L^p $-regularity, finite element method

\section{Introduction}
Stochastic partial differential equations (SPDEs) hold significant importance in both pure mathematics and various natural sciences, including physics, chemistry, biology, and engineering; see, e.g., \cite{Flandoli2023,Kolokoltsov2000,Lototsky2017,Pardoux2021,Prato2014}. Among these, stochastic parabolic equations constitute one of the most prominent classes of SPDEs. Over the past three decades, the numerical analysis of stochastic parabolic equations has garnered considerable attention;
see, e.g., \cite{Brehier2019IMA,Gyongy1999,Gyongy2003,Gyongy2009,Gyongy1997,Gyongy2016,Hausenblas2003,
Hutzenthaler2018,Kovacs2018-2,Kruse2014book,LiuQiao2021,Sauer2015,Wangxiaojie2020} (this list is by no means exhaustive).




While significant advancements have been made in the numerical analysis of stochastic parabolic equations,
existing research predominantly focuses on spatial $ L^2 $-norm settings, with remarkably few results available for general spatial $ L^q $-norms.
Foundational work by Cox and van Neerven \cite{Cox2010,Cox2013} laid the groundwork through their analysis of splitting schemes and Euler approximations for stochastic Cauchy problems in UMD Banach spaces. Blömker and Jentzen \cite{Blomker2013} studied Galerkin approximations of a one-dimensional stochastic Burgers equation in the spatial $ L^\infty $-norm framework. Subsequent developments by van Neerven and Veraar \cite{Neerven2022} established pathwise uniform convergence rates for temporal discretizations in 2-smooth Banach spaces, and more recently, Klioba and Veraar \cite{Klioba2024} extended these results to equations with irregular nonlinearities within the same setting.
Very recently, Li and Xie \cite{LiXieLpTime2025} established the discrete stochastic maximal $ L^p $-regularity estimate for the Euler scheme.
Despite these advances, a systematic theory for the convergence of finite element-based spatial discretizations of SPDEs in general Banach spaces remains
notably lacking in the literature. It is this gap that motivates the present study.

The principal contributions of this work are threefold.
First, using Dunford calculus and standard properties of the negative discrete Laplacian $\mathcal{A}_h$, discretized via the linear finite element method, we prove that the boundedness constant of the $H^\infty$-calculus of $\mathcal{A}_h$ is uniform with respect to the mesh size $h$. This result is further used to show that the norms of the purely imaginary powers of $\mathcal{A}_h$, as well as the $\mathcal{R}$-boundedness constant of its resolvent family, are also independent of $h$. These findings are stated in \Cref{thm:Ah-Hinf}.

Second, by combining the $h$-uniform boundedness of the $H^\infty$-calculus of $\mathcal{A}_h$ with \cite[Theorem~3.5]{Neerven2012b}, we establish a discrete stochastic maximal $L^p$-regularity estimate. Specifically, for any $p \in (2,\infty)$ and $q \in [2,\infty)$, there exists a constant $C > 0$ (independent of $h$) such that for all $f_{h} \in L_{\mathbb{F}}^{p}(\Omega \times \mathbb{R}_{+}; \gamma(H, X_{0,h}^{q}))$, the process $y_{h}$ defined by
\[
y_{h}(t) = \int_{0}^{t} e^{(s-t)\mathcal{A}_{h}} f_{h}(s) \,  \mathrm{d}W_{H}(s), \quad t \geqslant 0,
\]
satisfies
\[
\begin{aligned}
\Bigl( \mathbb{E} \sup_{t \geqslant 0} \, \lVert y_{h}(t) \rVert_{(X_{0,h}^{q}, X_{1,h}^{q})_{1/2 - 1/p, p}}^{p} \Bigr)^{1/p} &+ \lVert y_{h} \rVert_{L^{p}(\Omega \times \mathbb{R}_{+}; X_{1/2,h}^{q})} 
\leqslant C \lVert f_{h} \rVert_{L^{p}(\Omega \times \mathbb{R}_{+}; \gamma(H, X_{0,h}^{q}))}.
\end{aligned}
\]
Here, $H$ is a separable Hilbert space, $W_H$ is an $H$-cylindrical Brownian motion, and the scales $X_{\theta,h}^{q}$ are defined below. This estimate extends the deterministic discrete maximal $L^p$-regularity result from \cite{Geissert2006} to stochastic evolution equations and provides a key tool for the numerical analysis of spatial semi-discretizations of nonlinear stochastic parabolic problems. The precise statement is given in
\cref{thm:DSMLP}.

Third, for a spatially semi-discretized linear stochastic heat equation, we prove an optimal error estimate in the norm $\lVert \cdot \rVert_{L^p(\Omega \times \mathbb{R}_{+}; L^q(\mathcal{O}))}$ and a nearly optimal pathwise uniform error estimate in the norm of $L^p(\Omega; \mathrm{C}([0,\infty); L^q(\mathcal{O})))$, where $p \in (2,\infty)$ and $q \in [2,\infty)$. These convergence results are presented in \Cref{thm:conv}. They extend the numerical analysis of spatial semi-discretizations for stochastic evolution equations from the Hilbert space setting (see, e.g., \cite[Section~3.4]{Kruse2014book}) to the Banach space setting.

The remainder of this paper is organized as follows.  
Section \ref{sec:pre} introduces the notational conventions and fundamental concepts that will be used throughout the study.  
Section \ref{sec:main} presents the main results, which include:  
(i) a spatial mesh size-independent bounded $ H^\infty $-calculus for the negative discrete Laplace operator;  
(ii) discrete stochastic maximal $ L^p $-regularity; and  
(iii) a nearly optimal pathwise uniform convergence estimate.  
Section \ref{sec:proof} contains the proofs of these results.  
Finally, Section \ref{sec:conclusion} summarizes the contributions of this work.

\section{Preliminaries}
\label{sec:pre}

\medskip\noindent\textbf{Notational conventions.}
Throughout this paper, we adhere to the following notational conventions.
\begin{itemize}
  \item For any Banach spaces \( E_1 \) and \( E_2 \), the notation \(\mathcal{L}(E_1, E_2)\) represents the space of all bounded linear operators mapping from \( E_1 \) to \( E_2 \). When \( E_1 = E_2 \), we simplify the notation to \(\mathcal{L}(E_1)\).
    The resolvent set of an operator \( B \) is denoted by \(\rho(B)\).
    The identity mapping is denoted by $ I $.
  \item For real numbers $ \alpha \in (0, 1) $ and $ p \in (1, \infty) $, we denote by $ (\cdot, \cdot)_{\alpha, p} $
    the real interpolation space defined via the $ K $-method; see, e.g., \cite[Chapter~1]{Lunardi2018}.
    For $ \theta \in [0, 1] $, the complex interpolation space is denoted by $ [\cdot, \cdot]_\theta $,
    as detailed in \cite[Chapter~2]{Lunardi2018}.
  \item The conjugate exponent of any \( q \in [1,\infty] \), namely $ q/(q-1) $, is represented by \( q' \),
    and the symbol \( i \) stands for the imaginary unit.
  \item For any angle \(\theta \in (0, \pi)\), the sector \(\Sigma_{\theta}\) is defined
    as the set \(\{z \in \mathbb{C} \setminus \{0\} \mid |\operatorname{Arg}(z)| < \theta\}\),
    where \(\operatorname{Arg}(z)\) denotes the principal value of the argument of \(z\) within the interval \((- \pi, \pi]\). The boundary \(\partial\Sigma_{\theta}\) is oriented in such a manner that the positive real axis lies to its left.

  \item Suppose that \( (\Omega, \mathcal{F}, \mathbb{P}) \) is a complete probability space equipped with a filtration \( \mathbb{F} := (\mathcal{F}_t)_{t \geqslant 0} \) satisfying the usual conditions, on which a sequence of independent real-valued \( \mathbb{F} \)-Brownian motions \( (\beta_n)_{n \in \mathbb{N}} \) is defined. Throughout, the expectation operator with respect to the probability space \( (\Omega, \mathcal{F}, \mathbb{P}) \) will be denoted by \( \mathbb{E} \), unless otherwise indicated.

The mathematical content remains unchanged and correct.
  \item Let \( H \) be a separable Hilbert space, endowed with the inner product
    \( (\cdot, \cdot)_H \) and an orthonormal basis \( (h_n)_{n \in \mathbb{N}} \).
    We define the $ H $-cylindrical Brownian motion $ W_H $ such that, for each \( t \in \mathbb{R}_{+} \),
    \( W_H(t) \in \mathcal{L}(H, L^2(\Omega)) \) is given by
    \[
      W_H(t)x := \sum_{n \in \mathbb{N}} \beta_n(t) (x, h_n)_H, \quad \forall x \in H.
    \]
    For a detailed exposition on the theory of stochastic integration with
    respect to the cylindrical Brownian motion \( W_H \),
    we refer the reader to \cite{Neerven2007}.
  \item For any \( p \in (1, \infty) \) and any Banach space \( E \),
    let \( L_\mathbb{F}^p(\Omega \times \mathbb R_{+}; E) \) denote the
    collection of all \( \mathbb{F} \)-adapted \( E \)-valued processes
    that belong to \( L^p(\Omega \times \mathbb R_{+}; E) \).
    Furthermore, let \( \gamma(H, E) \) denote the space of all \( \gamma \)-radonifying
    operators from \( H \) to \( E \) (see \cite[Chapter~9]{HytonenWeis2017}).
  \item Let $\mathcal{O} \subset \mathbb{R}^3$ be a bounded convex domain with a boundary $ \partial\mathcal O $
    of class $ C^{2,0} $ (cf.~\cite[Section~6.2]{Gilbarg2001}). For any $ q \in [1,\infty] $, $ W_0^{1,q}(\mathcal O) $
    and $ W^{2,q}(\mathcal O) $ denote two standard Sobolev spaces (cf.~\cite[Chapter 7]{Gilbarg2001}.
  \item The symbol \( C_{\times} \) represents a positive constant that depends only on its subscript(s).
    Its specific value may vary from one instance to another.
\end{itemize}

\medskip\noindent\textbf{$ \mathcal R $-boundedness.}
Let $E_1$ and $E_2$ be Banach spaces. A family $\mathcal{A} \subset \mathcal{L}(E_1,E_2)$ is
called $\mathcal{R}$-bounded if there exists $C>0$ such that for every $N\in\mathbb{N}$, sequences $(B_n)_{n=1}^N \subseteq \mathcal{A}$, $(x_n)_{n=1}^N \subseteq E_1$, and independent symmetric $\{-1,1\}$-valued random variables $(r_n)_{n=1}^N$,
the following inequality holds:
\[
\mathbb{E} \Bigl\| \sum_{n=1}^N r_n B_n x_n \Bigr\|_{E_2}^2 \leqslant C\, \mathbb{E} \Bigl\| \sum_{n=1}^N r_n x_n \Bigr\|_{E_1}^2.
\]
Here, $\mathbb{E}$ denotes expectation over the probability space supporting $(r_n)_{n=1}^N$.
The smallest such constant $C$ is termed the $\mathcal R$-boundedness constant of $ \mathcal A $.
For further foundational results, we refer to \cite[Chapter 4]{Pruss2016}.


\medskip\noindent\textbf{$ H^\infty $-calculus.}
Let $ A $ be a sectorial operator on a Banach space $ E $.
For any $ \theta \in (0,\pi) $, we say that $ A $ admits a bounded $ H^\infty(\Sigma_\theta) $-calculus
if there exist constant $ C > 0 $ such that
$$
  \left\lVert \varphi(A) \right\rVert_{\mathcal{L}(E)} \leqslant C \sup_{z \in \Sigma_\theta} |\varphi(z)|
$$
for all $ \varphi \in \mathcal{H}_0^\infty(\Sigma_\theta) $, where the operator $ \varphi(A) $ is defined via the Dunford integral
$$
  \varphi(A) := \frac{1}{2\pi i} \int_{\partial \Sigma_\theta} \varphi(z) (z - A)^{-1} \, \mathrm{d}z,
$$
and the function space $ \mathcal{H}_0^\infty(\Sigma_\theta) $ consists of all analytic functions $ \varphi \colon \Sigma_\theta \to \mathbb{C} $ for which there exists some $ \varepsilon > 0 $ satisfying
$$
  \sup_{z \in \Sigma_\theta}
  \left( \frac{1 + | z |^2}{| z |} \right)^\varepsilon |\varphi(z)| < \infty.
$$
The smallest such constant $ C $ is called the boundedness constant of the $ H^\infty(\Sigma_\theta) $-calculus for $ A $.
For any $ \varphi \in \mathcal{H}_0^\infty(\Sigma_\theta) $, we equip this space with the norm
$$
  \|\varphi\|_{\mathcal{H}_0^\infty(\Sigma_{\theta})} :=
  \sup_{z \in \Sigma_\theta} |\varphi(z)|.
$$
We refer the reader to \cite[Chapter~10]{HytonenWeis2017} for a comprehensive treatment of $ H^\infty $-calculus.

\medskip\noindent\textbf{Definition of $A_q$.}
For any \(q \in (1, \infty)\), we denote by \(A_q\) the realization of the negative Laplacian \(-\Delta\) in \(L^q(\mathcal{O})\) with homogeneous Dirichlet boundary conditions.
The space \(L^q(\mathcal{O})\) is denoted by \(X_0^q\). For each \(\alpha \in (0, 1]\), \(X_\alpha^q\) represents the domain of \(A_q^\alpha\), equipped with the norm  
\[
\|v\|_{X_\alpha^q} := \|A_q^\alpha v\|_{X_0^q}, \quad \forall v \in X_\alpha^q.
\]  
For each \(\alpha \in [0, 1]\), \(X_{-\alpha}^{q'}\) denotes the dual space of \(X_\alpha^q\). It is well-known that \(A_q\) uniquely extends to a bounded linear operator from \(X_\alpha^q\) to \(X_{\alpha-1}^q\), satisfying  
\[
\|A_q v\|_{X_{\alpha-1}^q} = \|v\|_{X_\alpha^q}, \quad \forall v \in X_\alpha^q.
\]  
The space $ X_1^q $ coincides with $ W_0^{1,q}(\mathcal{O}) \cap W^{2,q}(\mathcal{O}) $,
and its norm is equivalent to the standard norm on $ W^{2,q}(\mathcal{O}) $ (see, e.g., \cite[Theorem~9.15]{Gilbarg2001}).
By \cite[Theorem~2]{Duong1996} and \cite[Theorem~16.5]{Yagi2010},
for any \(0 < \alpha < 1\), \(X_\alpha^q\) coincides with the complex interpolation space \([X_0^q, X_1^q]_\alpha\),
with equivalent norms. Furthermore, by \cite[Theorem~16.15]{Yagi2010}, \(X_{1/2}^q\) corresponds to the Sobolev space \(W_0^{1,q}(\mathcal{O})\), ensuring the validity of the identity  
\begin{equation}
  \label{eq:Aq-W12}
\langle A_q u, v \rangle = \int_{\mathcal{O}} \nabla u \cdot \nabla v \, \mathrm{d}x, \quad \text{for all } u \in X_{1/2}^q \text{ and } v \in X_{1/2}^{q'}.
\end{equation}
where \(\langle \cdot, \cdot \rangle\) denotes the duality pairing between \(X_{-1/2}^q\) and \(X_{1/2}^{q'}\).



\medskip\noindent\textbf{Definition of $\mathcal{A}_h$.}
Let \( \mathcal{K}_h \) be a conforming quasi-uniform tetrahedral mesh of
the domain \( \mathcal{O} \), where \( h \) denotes the maximum diameter of the elements in $ \mathcal K_h $.
We assume the triangulation satisfies the vertex compatibility condition: each vertex on \( \partial\mathcal{O}_h \) belongs to \( \partial\mathcal{O} \), where \( \mathcal{O}_h \) denotes the polyhedral domain induced by \( \mathcal{K}_h \). 
Define the finite element space
\[
  X_h := \left\{ v_h \in C(\overline{\mathcal{O}}) \,:\, v_h|_{\overline{\mathcal{O}}\setminus\mathcal{O}_h} = 0,
  \, \text{and $v_h$ is linear on each $ K \in \mathcal{K}_h$} \right\}.
\]
We consider the following quantities:
\begin{align*}
  \varrho_1 := \max_{K \in \mathcal{K}_h} \frac{h_K}{\rho_K}, \quad
  \varrho_2 := \frac{\max_{K \in \mathcal{K}_h} h_K}{\min_{K \in \mathcal{K}_h} h_K}, \quad
  \varrho := \{\varrho_1, \varrho_2\},
\end{align*}
where \( h_K \) and \( \rho_K \) represent the circumradius and inradius of
the element \( K \), respectively.
Let \( P_h \) denote the \( L^2(\mathcal{O}) \)-orthogonal projection operator onto \( X_h \).
This projection operator admits an extension to the space \( X_{-\alpha}^q \) for any \( \alpha \in (0, \frac{1}{2}] \) and \( q \in (1, \infty) \), which satisfies
\[
  \int_{\mathcal{O}} (P_hv) v_h \, \mathrm{d}x  
  = \langle v, \, v_h \rangle 
\]
for all \( v \in X_{-\alpha}^q \) and \( v_h \in X_h \).
Here, \(\langle \cdot, \cdot \rangle\) denotes the duality pairing between \(X_{-\alpha}^q\) and \(X_{\alpha}^{q'}\),
and we note that \(X_h \subset W_0^{1,q'}(\mathcal{O})\) implies \(X_h \subset X_\alpha^{q'}\) for all \(\alpha \in [0, 1/2]\).
For any \( u_h \in X_h \), the negative discrete Laplacian \( \mathcal{A}_h u_h \in X_h \) is defined by
\[
  \int_{\mathcal{O}} (\mathcal{A}_h u_h) \cdot v_h \, \mathrm{d}x
  = \int_{\mathcal{O}} \nabla u_h \cdot \nabla v_h \, \mathrm{d}x \quad \forall v_h \in X_h.
\]
It is evident that \( \mathcal{A}_h \) is a symmetric and positive-definite operator on \( X_h \),
where \(X_h\) is equipped with the $ L^2(\mathcal O)$-norm.
For any $ \alpha \in \mathbb R $ and $ q \in (1,\infty) $, let
$ X_{\alpha,h}^q $ denote the Banach space $ X_h $ equipped with the norm
\[
  \| v_h \|_{X_{\alpha,h}^q} := \| \mathcal{A}_h^\alpha v_h \|_{X_0^q}, \quad \forall v_h \in X_h.
\]
For further properties of \( \mathcal{A}_h \), we refer the reader to \cite{Bakaev2002} and the theory developed in \cite[Chapter~8]{Brenner2008}.

\section{Main Results}
\label{sec:main}

The first main result establishes key operator-theoretic properties of $\mathcal{A}_h$, as summarized in the following theorem.

\begin{theorem}
  \label{thm:Ah-Hinf}
  Let $ \theta \in (0,\pi/2) $ and $ q \in (1,\infty) $. The negative discrete Laplace operator
  $\mathcal{A}_h$ satisfies the following properties:
  \begin{enumerate}
    \item[\rm{(i)}] For any $ \phi \in \mathcal{H}_0^\infty(\Sigma_\theta) $,
      the following inequality holds:
      \begin{equation}
        \norm{\phi(\mathcal{A}_h)}_{\mathcal{L}(X^q_{0,h})} \leqslant
        C_{\theta,q,\varrho,\mathcal{O}} \norm{\phi}_{\mathcal{H}_0^\infty(\Sigma_\theta)}.
        \label{eq:Ah-Hinf}
      \end{equation}
    \item[\rm{(ii)}] For any $ t \in \mathbb{R} $, it holds that
      \begin{equation}
        \norm{\mathcal{A}_h^{it}}_{\mathcal{L}(X_{0,h}^q)}
        \leqslant C_{\theta,q,\varrho,\mathcal{O}} e^{\theta\abs{t}}.
        \label{eq:Ah-BIP}
      \end{equation}
    \item[\rm{(iii)}] The operator family $ \{z(z - \mathcal{A}_h)^{-1} \}_{z \in \mathbb{C} \setminus \overline{\Sigma_\theta}} $
      is $\mathcal{R}$-bounded on $ X_{0,h}^q $,
      and its $\mathcal{R}$-boundedness constant is uniformly bounded with respect to $ h $.
  \end{enumerate}
\end{theorem}

Property~(i) of \cref{thm:Ah-Hinf} is essential for establishing the discrete stochastic maximal $ L^p $-regularity for semi-discrete and fully discrete finite element approximations of SPDEs, and it plays a particularly critical role in the fully discrete setting.
Let $ p \in (2,\infty) $, $ q \in [2,\infty) $, and fix $ T > 0 $. Define the time step $ \tau := T/J $ for a positive integer $ J $. Let $ (Y_j)_{j=0}^J $ denote the solution to the fully discrete Euler scheme:
\[
\begin{cases}
  Y_{j+1} - Y_j + \tau \mathcal{A}_h Y_{j+1} = \displaystyle \int_{j\tau}^{j\tau+\tau} f_h(t)  \, \mathrm{d}W_H(t), & 0 \leqslant j < J, \\[0.5em]
    Y_0 = 0.
\end{cases}
\]
Here, $ f_h \in L_\mathbb{F}^p(\Omega \times (0,T); \gamma(H, X_{0,h}^q)) $. By Remark~3.11 of \cite{LiXieLpTime2025},
we obtain the discrete stochastic maximal $ L^p $-regularity estimate
\[
\left( \mathbb{E} \left[ \tau \sum_{j=1}^J \| \mathcal{A}_h^{1/2} Y_j \|_{X_{0,h}^q}^p \right] \right)^{\!1/p} \leqslant C_{p,q,\varrho,\mathcal{O}} \| f_h \|_{L^p(\Omega \times (0,T); \gamma(H, X_{0,h}^q))}.
\]
Furthermore, following Remark~4.2 in \cite{LiXieLpTime2025}, the following discrete maximal inequality holds:
\begin{align*}
& \left( \mathbb{E} \left[ \max_{1 \leqslant j \leqslant J}
\| Y_j \|_{X_{0,h}^q}^p \right] \right)^{\!1/p} \\
  \leqslant{} &
  C_{\alpha,p,q,\varrho,\mathcal{O}} \left( \| \mathcal{A}_h^{-\alpha} f_h \|_{L^p(\Omega \times (0,T); \gamma(H, X_{0,h}^q))} + \tau^{\frac{1}{2} - \frac{1}{p}} \| f_h \|_{L^p(\Omega \times (0,T); \gamma(H, X_{0,h}^q))} \right),
\end{align*}
for any $ \alpha \in (0, \frac{1}{2} - \frac{1}{p}) $. These estimates are fundamental for the numerical analysis of nonlinear stochastic parabolic equations.

Property~(ii) of \cref{thm:Ah-Hinf} enables the application of \cite[Theorem~4.17]{Lunardi2018}, ensuring the norm equivalence
\[
    c_0 \| v_h \|_{X_{\theta,h}^q} \leqslant \| v_h \|_{[X_{0,h}^q, X_{1,h}^q]_\theta} \leqslant c_1 \| v_h \|_{X_{\theta,h}^q} \quad \forall v_h \in X_h,
\]
between the complex interpolation space $[X_{0,h}^q, X_{1,h}^q]_\theta$ and $X_{\theta,h}^q$ for all $\theta \in (0,1)$ and $q \in (1,\infty)$, with constants $c_0, c_1 > 0$ independent of $h$. This equivalence is essential for analyzing SPDEs with low spatial regularity.
For the linear finite element method, \cite[pp.~921--922]{Kemmochi2018} established Property~(ii) under an acuteness condition on the spatial mesh.

Property~(iii) of \cref{thm:Ah-Hinf} enables the application of Weis' fundamental result \cite[Theorem~4.2]{Weis2001}, establishing the discrete maximal $ L^p $-regularity estimate:
\[
  \left\|\mathcal{A}_h \int_0^{(\cdot)} e^{(s-(\cdot))\mathcal{A}_h} f_h(s) \, \mathrm{d}s\right\|_{L^p(\mathbb{R}_+; X_{0,h}^q)} 
  \leqslant
  C_{p,q,\varrho,\mathcal{O}} \|f_h\|_{L^p(\mathbb{R}_+; X_{0,h}^q)},
\]
for all $ f_h \in L^p(\mathbb{R}_+; X_{0,h}^q) $ with $ p, q \in (1,\infty) $,
where $ (e^{-t\mathcal{A}_h})_{t \geqslant 0} $ denotes the analytic semigroup generated by $ -\mathcal{A}_h $.
This estimate was originally established by Geissert \cite{Geissert2006} through detailed kernel estimates for the semigroup $ (e^{-t\mathcal{A}_h})_{t \geqslant 0} $. Our approach employs $ H^\infty $-calculus theory instead, providing an alternative proof.

The second main result establishes the discrete stochastic maximal $L^p$-regularity estimate for the stochastic convolution with the semigroup $ (e^{-t\mathcal{A}_h})_{t \geqslant 0} $.

\begin{theorem}
  \label{thm:DSMLP}
  Let $ p \in (2, \infty) $ and $ q \in [2, \infty) $.
  For any $ f_h \in L^p_\mathbb{F}(\Omega \times \mathbb{R}_+; \gamma(H, X_{0,h}^q)) $, the following discrete stochastic maximal $L^p$-regularity estimate holds:
  \begin{equation}
    \label{eq:DSMLP}
    \left( \mathbb{E} \sup_{t \in \mathbb{R}_+} \, \norm{y_h(t)}_{(X_{0,h}^q, X_{1,h}^q)_{{1/2 - 1/p}, {p}}}^p \right)^{\!1/p}
    + \norm{y_h}_{L^p(\Omega \times \mathbb{R}_+; X_{1/2,h}^q)} 
    \leqslant C_{p,q,\varrho,\mathcal{O}} \, \norm{f_h}_{L^p(\Omega \times \mathbb{R}_+; \gamma(H, X_{0,h}^q))},
  \end{equation}
  where
  \[
  y_h(t) := \int_0^t e^{(s - t)\mathcal{A}_h} f_h(s) \, \mathrm{d}W_H(s), \quad t \in \mathbb{R}_+.
  \]
\end{theorem}

As a concrete application of \cref{thm:DSMLP}, for any $ p \in (2,\infty) $ and $ q \in [2,\infty) $,
the following estimate holds:
\begin{align*}
  &\mathbb{E} \bigg[
    \sup_{t \in \mathbb{R}_+} \, \norm{y_h(t)}_{(X_{0,h}^q,X_{1,h}^q)_{1/2-1/p,p}}^p
    + \int_{\mathbb{R}_+} \|\mathcal{A}_h^{1/2} y_h(t)\|_{L^q(\mathcal{O})}^p \,\mathrm{d}t \bigg] \\
  \leqslant{}
  & C_{p,q,\varrho,\mathcal{O}} \,
  \mathbb{E} \bigg[ \int_{\mathbb{R}_+} \Big\| \Big( \sum_{n\in\mathbb{N}} |f_{h,n}(t)|^2 \Big)^{1/2} \Big\|_{L^q(\mathcal{O})}^p \,\mathrm{d}t \bigg],
\end{align*}
where $ (f_{h,n})_{n\in\mathbb{N}} $ is a sequence of processes in $ L_\mathbb{F}^p(\Omega\times\mathbb{R}_+;X_{0,h}^q) $ such that the right-hand side is finite. The process $ y_h $ is defined by the stochastic convolution
\[
y_h(t) := \sum_{n\in\mathbb{N}} \int_0^t e^{(s-t)\mathcal{A}_h} f_{h,n}(s) \, \mathrm{d}\beta_n(s), \quad t \geqslant 0,
\]
with $ (\beta_n)_{n\in\mathbb{N}} $ being a sequence of independent real-valued Brownian motions as introduced in \cref{sec:pre}. 

\begin{remark}
 We provide a direct application of the discrete stochastic maximal $ L^p $-regularity estimate (Theorem~\ref{thm:DSMLP}) to the numerical analysis of nonlinear stochastic parabolic equations. Fix a terminal time $ T > 0 $, and let $ (\beta_n)_{n \in \mathbb{N}} $ be independent real-valued Brownian motions as defined in Section~\ref{sec:pre}. Consider continuous functions $ (f_n)_{n \in \mathbb{N}} : \mathcal{O} \times \mathbb{R} \to \mathbb{R} $ satisfying the uniform growth condition
 \[
   \left( \sum_{n\in\mathbb N} |f_n(x, y)|^2 \right)^{1/2} \leqslant C(1 + |y|), \quad \forall x \in \mathcal O, \, \forall y \in \mathbb{R},
 \]
 where $ C > 0 $ is a constant independent of $ x $ and $ y $. The spatial semidiscretization is given by
 \[
   \begin{cases}
     \mathrm{d} y_h(t) + \mathcal{A}_h y_h(t) \, \mathrm{d} t = \displaystyle\sum_{n\in\mathbb N} P_h f_n(y_h(t)) \, \mathrm{d} \beta_n(t), & t \in [0, T], \\
     y_h(0) = 0,
   \end{cases}
 \]
 where $ f_n(y_h(t)) $ is an abbreviation for the function $ f_n(\cdot, y_h(t)) $. Fix $ p, q \in (2, \infty) $ such that $ (1 - 2/p)q > 3 $.  
 By Theorem~\ref{thm:DSMLP}, the stability of $ P_h $ on $ X_0^q $ (see Lemma~\ref{lem:Ph}(i)),
 the ideal property of $ \gamma $-radonifying operators \cite[Theorem~9.1.10]{HytonenWeis2017},
 the $ \gamma $-Fubini isomorphism theorem \cite[Theorem~9.4.8]{HytonenWeis2017},
 and the $ h $-uniform continuous embedding (whose verification is left to the reader)
 \[
   \big( X_{0,h}^q, X_{1,h}^q \big)_{1/2 - 1/p, p} \hookrightarrow L^\infty(\mathcal{O}),
 \]
 we obtain for any $ t \in [0, T] $:
 \begin{align*}
   \mathbb{E} \sup_{s \in [0,t]} \, \| y_h(s) \|_{L^\infty(\mathcal{O})}^p
   &\leqslant c \, \mathbb{E} \int_0^t \Big\| \Big( \sum_{n\in\mathbb N} |f_n(y_h(s))|^2 \Big)^{1/2} \Big\|_{L^q(\mathcal{O})}^p  \mathrm{d} s \\
   &\leqslant c \, \mathbb{E} \int_0^t \left( 1 + \sup_{r \in [0,s]} \, \| y_h(r) \|_{L^q(\mathcal{O})}^p \right) \mathrm{d} s \\
   &\leqslant c \, \mathbb{E} \int_0^t \left( 1 + \sup_{r \in [0,s]} \, \| y_h(r) \|_{L^\infty(\mathcal{O})}^p \right) \mathrm{d} s,
 \end{align*}
 where we used the boundedness of $ \mathcal{O} $,
 and $ c > 0 $ denotes a generic constant independent of $ h $ and $ t \in [0,T] $. An application of Gronwall’s inequality then yields the $ h $-uniform bound
 \[
   \sup_{h > 0} \,  \mathbb{E} \left[ \sup_{t \in [0,T]} \| y_h(t) \|_{L^\infty(\mathcal{O})}^p \right] < \infty.
 \]
\end{remark}

The third main result establishes the convergence of a spatial semidiscretization for a linear stochastic heat equation in general spatial $ L^q $-norms.
\begin{theorem}\label{thm:conv}
    Let $ p \in (2,\infty) $, $ q \in [2,\infty) $, and assume $ f \in L_\mathbb{F}^p(\Omega \times \mathbb{R}_{+}; \gamma(H, X_0^q)) $. Let $ y $ denote the mild solution to
    \begin{subequations}\label{eq:y}
      \begin{numcases}{}
            \mathrm{d}y(t) + A_q y(t) \, \mathrm{d}t = f(t) \, \mathrm{d}W_H(t), \qquad  t > 0, \\
            y(0) = 0, 
        \end{numcases}
    \end{subequations}
    and let $ y_h $ denote the mild solution to the spatial semidiscretization
    \begin{subequations}\label{eq:yh}
      \begin{numcases}{}
            \mathrm{d}y_h(t) + \mathcal{A}_h y_h(t) \, \mathrm{d}t = P_h f(t) \, \mathrm{d}W_H(t), \qquad t > 0, \\
            y_h(0) = 0.
        \end{numcases}
    \end{subequations}
    Then
    \begin{equation}\label{eq:conv1}
        \|y - y_h\|_{L^p(\Omega \times \mathbb{R}_{+}; X_0^q)} 
        \leqslant C_{p,q,\varrho,\mathcal{O}} \, h \, \|f\|_{L^p(\Omega \times \mathbb{R}_{+}; \gamma(H, X_0^q))}.
    \end{equation}
    Moreover, for any $ \alpha \in [0, 1/p] $,
    \begin{equation}\label{eq:conv2}
        \left[ \mathbb{E} \sup_{t \in \mathbb{R}_{+}} \|(y - y_h)(t)\|_{X_{-\alpha}^q}^p \right]^{1/p} 
        \leqslant C_{\alpha,p,q,\varrho,\mathcal{O}} \, \big(1 + |\ln h|^{1-1/p}\big) \, h^{1 + 2(\alpha - 1/p)} \, \|f\|_{L^p(\Omega \times \mathbb{R}_{+}; \gamma(H, X_0^q))}.
    \end{equation}
  These estimates remain valid on any finite time interval \([0,T]\), \(T>0\), with \(\mathbb{R}_{+}\) replaced by \([0,T]\).
\end{theorem}

For convergence results concerning spatial semidiscretizations of stochastic parabolic equations in the Hilbert space setting, see \cite[Section~3.4]{Kruse2014book}.
Notably, the case $p=2$, $q>2 $ lies beyond the scope of the present analysis.

\section{Proofs}
\label{sec:proof}

\subsection{Preliminary Lemmas}
\label{ssec:preliminary}

We collect in this section several standard estimates that are well established in the literature. 
The following fundamental properties of the negative Laplace operator $A_q$ are drawn from \cite[Theorem~2.12]{Yagi2010}, \cite[Theorem~2]{Duong1996}, and \cite[Proposition~9.10]{Kalton2006}.

\begin{lemma}
  \label{lem:Aq}
  Let $\theta \in (0, \pi/2)$ and $q \in (1, \infty)$. Then:
  \begin{enumerate}
    \item[\rm{(i)}] For all $z \in \mathbb{C} \setminus \Sigma_\theta$:
      \begin{align*}
   & \|(z - A_q)^{-1}\|_{\mathcal{L}(X_0^q)} \leqslant \frac{C_{\theta,q,\mathcal{O}}}{1 + |z|}, \\
   & \|A_q(z - A_q)^{-1}\|_{\mathcal{L}(X_0^q)} \leqslant C_{\theta,q,\mathcal{O}}.
      \end{align*}
    \item[\rm{(ii)}] For any $\phi \in \mathcal{H}_0^\infty(\Sigma_\theta)$:
      \[
        \|\phi(A_q)\|_{\mathcal{L}(X_0^q)} \leqslant C_{\theta,q,\mathcal{O}} \|\phi\|_{\mathcal{H}_0^\infty(\Sigma_\theta)}.
      \]
  \end{enumerate}
\end{lemma}

The projection operator $ P_h $ has the following standard
stability and approximation property; see, e.g., \cite{Douglas1975}.
\begin{lemma}
  \label{lem:Ph}
  The operator $P_h$ satisfies:
  \begin{enumerate}
    \item[\rm{(i)}] For any $ q \in [1,\infty] $, $\|P_h\|_{\mathcal{L}(X_0^q,X_{0,h}^q)}$ is uniformly bounded with respect to $h$,
      though it may depend on $q$, $ \varrho $, and $ \mathcal O $.
    \item[\rm{(ii)}]
      For any $\alpha \in [0,1]$ and $q \in (1,\infty)$,
      $\|I - P_h\|_{\mathcal{L}(X_\alpha^q,X_0^q)} \leqslant C_{\alpha,q,\varrho,\mathcal O} h^{2\alpha}$.
    \item[\rm{(iii)}] For any $ q \in (1,\infty) $, $ \norm{\nabla(I-P_h)v}_{L^q(\mathcal O_h;\mathbb R^3)} \leqslant C_{q,\varrho,\mathcal O}h\norm{v}_{X_1^q} $,
      $ \forall v \in X_1^q $.
  \end{enumerate}
\end{lemma}

The negative discrete Laplace operator $ \mathcal{A}_h $ has the following important properties. 
\begin{lemma} 
  \label{lem:z-Ah-inv}
  Suppose that $ \theta \in (0, \frac{\pi}{2}) $ and $ q \in (1, \infty) $.
  Then, for all $ z \in \mathbb{C} \setminus \Sigma_\theta $, the following inequalities hold:
  \begin{align}
    & \| (z - \mathcal{A}_h)^{-1} \|_{\mathcal{L}(X^q_{0,h})}
    \leqslant \frac{C_{\theta,q,\varrho,\mathcal{O}}}{1 + |z|},
    \label{eq:z-Ah-inv} \\
    & \| \mathcal{A}_h(z - \mathcal{A}_h)^{-1} \|_{\mathcal{L}(X^q_{0,h})}
    \leqslant C_{\theta,q,\varrho,\mathcal{O}}.
    \label{eq:Ah*z-Ah}
  \end{align}
\end{lemma}
\begin{proof}
  Theorem 1.1 of \cite{Bakaev2002}, combined with Lemma~\ref{lem:Ph}(i), establishes
  the following assertion for $ s = \infty $:
  \begin{equation}
    \label{eq:tmp}
    \| (z - \mathcal{A}_h)^{-1} P_h \|_{\mathcal{L}(L^s(\mathcal{O}))}
    \leqslant \frac{C_{\theta,\varrho,\mathcal{O}}}{1 + |z|},
    \quad \forall z \in \mathbb{C} \setminus \Sigma_{\theta}.
  \end{equation}
  The operator $ \mathcal{A}_h $ is symmetric and positive definite on $ X_{0,h}^2 $,
  with its smallest eigenvalue bounded below uniformly in $ h $.
  This immediately verifies \eqref{eq:tmp} for $ s = 2 $.
  Interpolation then extends the result to $ s \in (2, \infty) $,
  while a duality argument extends it to $ s \in (1, 2) $.
  Since $ P_h $ acts as the identity on $ X_{0,h}^q $, \eqref{eq:z-Ah-inv} follows.
  The inequality \eqref{eq:Ah*z-Ah} is then immediate from \eqref{eq:z-Ah-inv}
  via the resolvent identity.
\end{proof}

\begin{lemma}
  \label{lem:etAh}
  Let $ 0 \leqslant \alpha \leqslant 1 $ and $ q \in (1,\infty) $. The following estimates hold:
  \begin{enumerate}
    \item[\rm{(i)}] For all $ t >0 $, 
      \begin{equation}
        \|\mathcal{A}_h^\alpha e^{-t\mathcal{A}_h}\|_{\mathcal{L}(X_{0,h}^q)}
        \leqslant C_{\alpha,q,\varrho,\mathcal{O}}
        t^{-\alpha}e^{-\delta t}, \quad \forall t > 0,
      \end{equation}
      where $ \delta $ is a positive constant that depends only on $ q $, $ \varrho $, and $ \mathcal{O} $.
    \item[\rm{(ii)}] 
      $\|\mathcal{A}_h^\alpha\|_{\mathcal{L}(X_{0,h}^q)} \leqslant C_{\alpha,q,\varrho,\mathcal{O}} h^{-2\alpha}$.
  \end{enumerate}
\end{lemma}
\begin{proof}
  Assertion (i) follows from Theorem~6.13 in \cite[Chapter 2]{Pazy1983}, combined with the resolvent estimate \eqref{eq:z-Ah-inv} in Lemma~\ref{lem:z-Ah-inv}.
  For (ii), Theorem~6.10 in \cite[Chapter 2]{Pazy1983} and the standard inverse estimate
  $$
    \|\mathcal{A}_h v_h\|_{X_{0,h}^q} \leqslant C_{q,\varrho,\mathcal{O}} h^{-2} \|v_h\|_{X_{0,h}^q}, \quad \forall v_h \in X_h,
  $$
  yield the required inequality.
\end{proof}

For any $ q \in (1,\infty) $, let $ R_h: X_{1/2}^q \to X_h $ denote the standard Ritz projection, defined by
\[
  \int_{\mathcal{O}_h} \nabla (u - R_h u) \cdot \nabla v_h \, \mathrm{d}x = 0, \quad \forall u \in X_{1/2}^q, \, v_h \in X_h.
\]
It is clear that $R_hv = \mathcal{A}_h^{-1}P_hA_qv $ holds for all $ v \in X_{1/2}^q $.
\begin{lemma}
  \label{lem:Rh}
  The operator $ R_h $ satisfies the following properties:
  \begin{enumerate}
    \item[\rm{(i)}] For any $ q \in (1,\infty) $,
      \begin{equation} 
        \norm{I - R_h}_{\mathcal{L}(X_1^q, X_0^q)} \leqslant C_{q,\varrho,\mathcal{O}} h^{2}.
        \label{eq:Rh-conv-1} 
      \end{equation}
    \item[\rm{(ii)}] For any $ q \in [2,\infty) $,
      \begin{equation} 
        \norm{I - R_h}_{\mathcal{L}(X_{1/2}^q, X_0^q)} \leqslant C_{q,\varrho,\mathcal{O}} h.
        \label{eq:Rh-conv-2} 
      \end{equation}
  \end{enumerate}
\end{lemma}
\begin{proof}
  These properties are well-known in the literature.
  They follow directly from the theory in \cite{Bakaev2002} combined with standard finite element error analysis
  for second-order elliptic equations (see, e.g., \cite[Chapter 8]{Brenner2008}).
  For completeness, we sketch the argument, which comprises four steps.

\textbf{Step 1.}
We begin by establishing the stability estimate
\begin{equation}
    \label{eq:Rhv-W1q-stab}
    \norm{\nabla R_h v}_{L^q(\mathcal{O}_h; \mathbb{R}^3)}
    \leqslant C_{q,\varrho,\mathcal O} \norm{\nabla v}_{L^q(\mathcal{O}_h; \mathbb{R}^3)},
    \quad \forall v \in X_{1/2}^q, \, \forall q \in (1, \infty).
\end{equation}
By \cite[Equation~(2.5)]{Bakaev2002}, we have the uniform bound
$$
    \norm{\nabla R_h v}_{L^\infty(\mathcal{O}_h; \mathbb{R}^3)} \leqslant C_{\varrho,\mathcal{O}} \norm{\nabla v}_{L^\infty(\mathcal{O}; \mathbb{R}^3)}, \quad \forall v \in W_0^{1,\infty}(\mathcal{O}).
$$
Moreover, from the definition of $ R_h $, it follows that
\begin{equation}
    \label{eq:Rhv-W12}
    \norm{\nabla R_h v}_{L^2(\mathcal{O}_h; \mathbb{R}^3)} \leqslant
    \norm{\nabla v}_{L^2(\mathcal{O}_h; \mathbb{R}^3)}
    \leqslant \norm{\nabla v}_{L^2(\mathcal{O}; \mathbb{R}^3)},
    \quad \forall v \in W_0^{1,2}(\mathcal O).
\end{equation}
Using complex interpolation, we deduce that
\begin{equation}
  \label{eq:Rhv-1}
  \| \nabla R_h v \|_{L^q(\mathcal{O}_h; \mathbb{R}^3)}
  \leqslant C_{q,\varrho,\mathcal{O}} \norm{\nabla v}_{L^q(\mathcal O;\mathbb R^3)},
  \quad \forall v \in W_0^{1,q}, \, \forall q \in [2, \infty).
\end{equation}
Next, consider $q \in (1, 2)$. For $v \in X_{1/2}^q$, a duality argument gives
\begin{align*}
  & \|\nabla R_h v\|_{L^q(\mathcal{O}_h; \mathbb{R}^3)}
  \leqslant C_{\mathcal{O}} \|R_h v\|_{X_{1/2}^q}
     = C_{\mathcal{O}} \sup_{0 \neq g \in X_{-1/2}^{q'} } 
    \frac{\dual{R_h v, g}}{\|g\|_{X_{-1/2}^{q'}}} \\
    & = C_{\mathcal{O}} \sup_{0 \neq g \in X_{-1/2}^{q'}} 
    \frac{\int_{\mathcal O}\nabla R_h v \cdot  \nabla\big(A_{q'}^{-1} g\big) \,  \mathrm{d}x}{\|g\|_{X_{-1/2}^{q'}}} \quad \text{(by \cref{eq:Aq-W12})} \\
    & = C_{\mathcal{O}} \sup_{0 \neq g \in X_{-1/2}^{q'}} 
    \frac{\int_{\mathcal O}\nabla v \cdot \nabla \big( R_h A_{q'}^{-1} g \big) \,  \mathrm{d}x}{\|g\|_{X_{-1/2}^{q'}}} \quad \text{(by the definition of $R_h$)},
\end{align*}
where $ \dual{\cdot,\cdot} $ denotes the duality pairing between $ X_{1/2}^q $ and $ X_{-1/2}^{q'} $.
Since $ \operatorname{supp}(R_h A_{q'}^{-1} g) \subset \overline{\mathcal{O}_h} $, we can restrict the integration to $ \mathcal{O}_h $, yielding
\begin{align*}
    \norm{\nabla R_h v}_{L^q(\mathcal{O}_h; \mathbb{R}^3)}
    & \leqslant C_{\mathcal{O}} \sup_{0 \neq g \in X_{-1/2}^{q'}} 
    \frac{\int_{\mathcal O_h}\nabla v \cdot \nabla \big( R_h A_{q'}^{-1} g\big) \, \mathrm{d}x}{\norm{g}_{X_{-1/2}^{q'}}}.
\end{align*}
Applying Hölder's inequality, estimate \eqref{eq:Rhv-1}, and the equality $ \norm{A_{q'}^{-1}g}_{X_{1/2}^{q'}} = \norm{g}_{X_{-1/2}^{q'}} $, we obtain
\begin{equation}
    \label{eq:Rhv-2}
    \norm{\nabla R_h v}_{L^q(\mathcal{O}_h; \mathbb{R}^3)}
    \leqslant C_{q,\varrho,\mathcal{O}} \norm{\nabla v}_{L^q(\mathcal{O}_h; \mathbb{R}^3)},
    \quad \forall v \in X_{1/2}^q, \, \forall q \in (1, 2).
\end{equation}
A similar duality argument using \eqref{eq:Rhv-2} yields
\begin{equation}
    \label{eq:Rhv-3}
    \norm{\nabla R_h v}_{L^q(\mathcal{O}_h; \mathbb{R}^3)}
    \leqslant C_{q,\varrho,\mathcal{O}} \norm{\nabla v}_{L^q(\mathcal{O}_h; \mathbb{R}^3)},
    \quad \forall v \in X_{1/2}^q, \, \forall q \in (2, \infty).
\end{equation}
Combining estimates \eqref{eq:Rhv-W12}, \eqref{eq:Rhv-2}, and \eqref{eq:Rhv-3}, we conclude the desired stability estimate \eqref{eq:Rhv-W1q-stab}.

\textbf{Step 2.} 
Using the stability estimate \eqref{eq:Rhv-W1q-stab} and the identity $ R_h v_h = v_h $ for all $ v_h \in X_h $, we obtain, for any $ v \in X_1^q $ and $ q \in (1,\infty) $,
\begin{align*}
    \norm{\nabla (I - R_h) v}_{L^q(\mathcal{O}_h; \mathbb{R}^3)}
    &= \norm{\nabla (I - R_h) (v - P_h v)}_{L^q(\mathcal{O}_h; \mathbb{R}^3)} \\
    &\leqslant C_{q,\varrho,\mathcal{O}} \norm{\nabla (v - P_h v)}_{L^q(\mathcal{O}_h; \mathbb{R}^3)}.
\end{align*}
An application of Lemma~\ref{lem:Ph}(iii) then yields
\begin{equation}
    \label{eq:Rh-conv-W1q}
    \| \nabla (v - R_h v) \|_{L^q(\mathcal{O}_h; \mathbb{R}^3)}
    \leqslant C_{q,\varrho,\mathcal{O}} h \norm{v}_{X_1^q},
    \quad \forall v \in X_1^q, \quad \forall q \in (1, \infty).
\end{equation}

\textbf{Step 3.}
We now establish property (i). Recall that $\mathcal{O} \subset \mathbb{R}^3$ is a bounded convex domain with a $C^{2,0}$ boundary $\partial\mathcal{O}$, and that each vertex of $\partial\mathcal{O}_h$ lies on $\partial\mathcal{O}$. Consequently, $\mathcal{O}_h$ is a convex polyhedral approximation of $\mathcal{O}$ satisfying
\begin{equation}
  \label{eq:cond-Kh}
  \max_{x \in \partial\mathcal{O}_h} \mathrm{dist}(x, \partial\mathcal{O}) \leqslant C_{\varrho,\mathcal{O}} h^2.
\end{equation}
Fix $ q \in (3, \infty) $, and let $ v \in X_1^q $, $ g \in X_0^{q'} $.
Define $ w := A_{q'}^{-1} g $.
By \cref{eq:Aq-W12}, the definition of $ R_h $, and the fact $ R_h v = P_hw = 0 $ in $\mathcal O\setminus\mathcal O_h $, we have
\begin{align*}
\int_{\mathcal{O}} (v - R_h v) g \, \mathrm{d}x 
& = \int_{\mathcal O} \nabla(v-R_hv) \cdot \nabla w \, \mathrm{d}x \\
& = \int_{\mathcal O} \nabla(v-R_hv) \cdot \nabla(w-P_hw) \, \mathrm{d}x \\
&= \int_{\mathcal{O} \setminus \mathcal{O}_h} \nabla v \cdot \nabla w \, \mathrm{d}x + \int_{\mathcal{O}_h} \nabla(v - R_h v) \cdot \nabla(w - P_h w) \, \mathrm{d}x \\
&=: I_1 + I_2.
\end{align*}
Applying integration by parts to $I_1$, we obtain
\begin{align*}
  I_1 & = \int_{\mathcal{O} \setminus \mathcal{O}_h} v(-\Delta w) \, \mathrm{d}x
  - \int_{\partial\mathcal{O}_h} v \frac{\partial w}{\partial\nu} \, \mathrm{d}\sigma,
\end{align*}
where $\frac{\partial w}{\partial\nu}$ denotes the outward normal derivative and $\mathrm{d}\sigma$ the surface measure on $\partial\mathcal{O}_h$.
By Hölder’s inequality and the trace inequality (see, e.g., \cite[Theorem III.2.36]{Boyer2012}), it follows that
\begin{align*}
  I_1 & \leqslant C_{\mathcal{O}} \left(  \| v \|_{L^q(\mathcal{O} \setminus \mathcal{O}_h)}\| \Delta w \|_{L^{q'}(\mathcal{O} \setminus \mathcal{O}_h)}  + \| v \|_{L^{2q/3}(\partial\mathcal{O}_h)} \| \nabla w \|_{L^{2q/(2q-3)}(\partial\mathcal{O}_h; \mathbb{R}^3)}  \right) \\
      & \leqslant C_{\mathcal{O}} \left( \| v \|_{L^q(\mathcal{O} \setminus \mathcal{O}_h)} + \| v \|_{L^{2q/3}(\partial\mathcal{O}_h)} \right) \| w \|_{X_1^{q'}}.
\end{align*}
Using \cref{eq:cond-Kh}, the boundary condition $ v = 0 $ on $ \partial\mathcal{O} $, and Sobolev's embedding theorem,
an elementary calculation yields
\begin{align*}
  \| v \|_{L^q(\mathcal{O} \setminus \mathcal{O}_h)}
  &\leqslant C_{\varrho,\mathcal{O}} h^2 \| \nabla v \|_{L^q(\mathcal{O} \setminus \mathcal{O}_h; \mathbb{R}^3)} \leqslant C_{\varrho,\mathcal{O}} h^2 \| v \|_{X_{1/2}^q}, \\
  \| v \|_{L^{2q/3}(\partial\mathcal{O}_h)}
  &\leqslant C_{\varrho,\mathcal O} h^2 \| \nabla v \|_{C(\overline{\mathcal{O}}; \mathbb{R}^3)}
  \leqslant C_{\varrho,\mathcal{O}} h^2 \| v \|_{X_1^q}.
\end{align*}
Combining these estimates, we conclude
$$
I_1 \leqslant C_{\varrho,\mathcal{O}} h^2 \| v \|_{X_1^q} \| w \|_{X_1^{q'}}.
$$
For $ I_2 $, applying \cref{eq:Rh-conv-W1q} and Lemma~\ref{lem:Ph}(iii), we deduce
$$
I_2 \leqslant \| \nabla(v - R_h v) \|_{L^q(\mathcal{O}_h; \mathbb{R}^3)} \| \nabla(w - P_h w) \|_{L^{q'}(\mathcal{O}_h; \mathbb{R}^3)} \leqslant C_{q,\varrho,\mathcal{O}} h^2 \| v \|_{X_1^q} \| w \|_{X_1^{q'}}.
$$
Therefore, combining both terms, we obtain
$$
\int_{\mathcal{O}} (v - R_h v) g \, \mathrm{d}x \leqslant C_{q,\varrho,\mathcal{O}} h^2 \| v \|_{X_1^q} \| w \|_{X_1^{q'}} \leqslant C_{q,\varrho,\mathcal{O}} h^2 \| v \|_{X_1^q} \| g \|_{X_0^{q'}},
$$
which holds for all $ v \in X_1^q $ and $ g \in X_0^{q'} $. This establishes \cref{eq:Rh-conv-1} for $ q \in (3, \infty) $.
Next, observe that the equality $ \| I - R_h \|_{\mathcal{L}(X_1^q, X_0^q)} = \| I - R_h \|_{\mathcal{L}(X_1^{q'}, X_0^{q'})} $ implies that \cref{eq:Rh-conv-1} also holds for $ q \in (1, 3/2) $. Interpolation then extends the result to $ q \in [3/2, 3] $, completing the proof of property (i).

\textbf{Step 4.}  
We now prove property (ii). Using \cref{eq:cond-Kh} and the boundary condition \( v = 0 \) on \( \partial\mathcal{O} \), one can show
\[
\| v \|_{L^{2q/3}(\partial\mathcal{O}_h)} \leqslant C_{q,\varrho,\mathcal{O}} h^{2 - 2/q} \| \nabla v \|_{L^q(\mathcal{O}; \mathbb{R}^3)}
\leqslant C_{q,\varrho,\mathcal{O}} h^{2 - 2/q} \| v \|_{X_{1/2}^q}, \quad \forall v \in X_{1/2}^q, \; \forall q \in (3/2, \infty).
\]
From \cref{eq:Rhv-W1q-stab}, we also have
\[
\| \nabla(v - R_h v) \|_{L^q(\mathcal{O}_h; \mathbb{R}^3)} \leqslant C_{q,\varrho,\mathcal{O}} \| v \|_{X_{1/2}^q}, \quad \forall v \in X_{1/2}^q, \; \forall q \in (1, \infty).
\]
With these estimates, a modification of the argument in Step 3 yields
\[
\int_{\mathcal{O}} (v - R_h v) g \, \mathrm{d}x \leqslant C_{q,\varrho,\mathcal{O}} h \| v \|_{X_{1/2}^q} \| g \|_{X_0^{q'}}, \quad \forall v \in X_{1/2}^q, \; \forall g \in X_0^{q'}, \; \forall q \in [2, \infty).
\]
Since this holds for all \( v \in X_{1/2}^q \) and \( g \in X_0^{q'} \), the desired estimate \eqref{eq:Rh-conv-2} follows. This completes the proof of (ii).

\end{proof}

Finally, we present a standard technical estimate.
\begin{lemma}
  \label{lem:z-Ah}
  Let $\theta \in (0, \tfrac{\pi}{2})$ and $q \in (1, \infty)$.
  Then, for every $z \in \mathbb{C} \setminus \Sigma_{\theta}$, we have
  \[
    \norm{(z - \mathcal{A}_h)^{-1} P_h - (z - A_q)^{-1}}_{\mathcal{L}(X_0^q)}
    \leqslant C_{\theta, q, \varrho, \mathcal{O}} h^2.
  \]
\end{lemma}
\begin{proof}
  We begin by noting that
  \begin{align*}
    \norm{(P_h-I)(z-A_q)^{-1}}_{\mathcal L(X_0^q)}
    & \leqslant
    \norm{P_h-I}_{\mathcal L(X_1^q,X_0^q)} 
    \norm{(z-A_q)^{-1}}_{\mathcal L(X_0^q,X_1^q)} \\
    & \leqslant
    C_{\theta,q,\varrho,\mathcal O} h^2,
  \end{align*}
  where the last inequality is justified by \cref{lem:Aq}(i) and \cref{lem:Ph}(ii).
  It suffices, therefore, to prove that
  \[
    \norm{(z-\mathcal{A}_h)^{-1}P_h - P_h(z-A_q)^{-1}}_{\mathcal L(X_0^q)}
    \leqslant C_{\theta,q,\varrho,\mathcal O} h^2.
  \]

  To this end, consider the following algebraic manipulation:
  \begin{align*}
    & (z-\mathcal{A}_h)^{-1}P_h - P_h(z-A_q)^{-1} \\
    ={}&
    (z-\mathcal{A}_h)^{-1}(P_h(z-A_q) - (z-\mathcal{A}_h)P_h)(z-A_q)^{-1} \\
    ={}&
    (z-\mathcal{A}_h)^{-1}(\mathcal{A}_hP_h - P_hA_q)(z-A_q)^{-1} \\
    ={}&
    (z-\mathcal{A}_h)^{-1}\mathcal{A}_h(P_hA_q^{-1} - \mathcal{A}_h^{-1}P_h)A_q(z-A_q)^{-1}.
  \end{align*}
  Subsequently,
  \begin{align*}
    & \norm{
      (z-\mathcal{A}_h)^{-1}P_h - P_h(z-A_q)^{-1}
    }_{\mathcal L(X_0^q)} \\
    \leqslant{}&
    \norm{(z-\mathcal{A}_h)^{-1}\mathcal{A}_h}_{\mathcal L(X_{0,h}^q)}
    \norm{P_hA_q^{-1} - \mathcal{A}_h^{-1}P_h}_{\mathcal L(X_0^q)}
    \norm{A_q(z-A_q)^{-1}}_{\mathcal L(X_0^q)} \\
    \stackrel{\text{(i)}}{\leqslant}{} &
    C_{\theta,q,\varrho,\mathcal O}
    \norm{P_hA_q^{-1} - \mathcal{A}_h^{-1}P_h}_{\mathcal L(X_0^q)} \\
    \stackrel{\text{(ii)}}{=}{} &
    C_{q,\varrho,\mathcal O}
    \norm{(P_h - R_h)A_q^{-1}}_{\mathcal L(X_0^q)} \\
    \leqslant{} & 
    C_{\theta,q,\varrho,\mathcal O} \Big(
    \norm{I-P_h}_{\mathcal L(X_1^q,X_0^q)}
    + \norm{I - R_h}_{\mathcal L(X_1^q,X_0^q)}
    \Big) \norm{A_q^{-1}}_{\mathcal L(X_0^q,X_1^q)} \\
    \stackrel{\text{(iii)}}{\leqslant}{} &
    C_{\theta,q,\varrho,\mathcal O} h^2,
  \end{align*}
  where step (i) follows from \cref{lem:Aq}(i) and
  inequality \cref{eq:Ah*z-Ah}, step (ii) exploits the identity $ R_h = \mathcal{A}_h^{-1}P_hA_q $,
  and step (iii) is a direct consequence of \cref{lem:Ph}(ii) and
  \cref{lem:Rh}(i),
  combined with the fact $ \norm{A_q^{-1}}_{\mathcal L(X_0^q,X_1^q)} = 1 $.
  This completes the proof.
\end{proof}

\subsection{Proof of \texorpdfstring{\cref{thm:Ah-Hinf}}{}}
\label{ssec:1}
The proof proceeds in seven steps: assertions (i) and (ii) follow from
Steps 1–5 and Step 6 respectively, while assertion (iii) is established in Step 7.
The argument relies on \cref{lem:Aq,lem:etAh,lem:z-Ah,lem:Ph,lem:etAh}.

\textbf{Step 1.} We prove that
\begin{equation} 
  \label{eq:wtAq-Hinf}
  \norm{\phi(\widetilde A_q)}_{\mathcal L(X_0^q)} \leqslant
  C_{\theta,q,\mathcal O} \norm{\phi}_{\mathcal H_0^\infty(\Sigma_\theta)}
  \quad\text{for all } \phi \in \mathcal H_0^\infty(\Sigma_\theta),
\end{equation}
where the operator \( \widetilde A_q \) is defined as
\begin{equation} 
  \label{eq:wtAq}
  \widetilde A_q := (h^2 + A_q) (1 + h^2 A_q)^{-1}.
\end{equation}
The introduction of \( \widetilde A_q \) is motivated by \cite[Section~3.1]{Pruss2016}.
Furthermore, according to Proposition~3.1.4 from \cite{Pruss2016},
$ \widetilde A_q $ is bounded, has a bounded inverse, and is a sectorial operator.
For any $ z \in \partial\Sigma_\theta $, a direct calculation gives
\begin{align*}
(z - \widetilde A_q)^{-1}
&= (1 + h^2 A_q) \left( z + h^2z A_q - h^2 - A_q \right)^{-1} \\
&= \frac{1}{1 - h^2z} (1 + h^2 A_q) \left( \frac{z - h^2}{1 - h^2z} - A_q \right)^{-1} \\
&=  \frac{h^2}{1 - h^2z} \left[
h^{-2} + \frac{z-h^2}{1-h^2z} - \left(\frac{z - h^2}{1 - h^2z} - A_q\right) 
\right] \left( \frac{z - h^2}{1 - h^2z} - A_q \right)^{-1} \\
&= \frac{h^2}{1 - h^2z} \left[
\frac{h^{-2} - h^2}{1 - h^2z} - \left( \frac{z - h^2}{1 - h^2z} - A_q \right) 
\right] \left( \frac{z - h^2}{1 - h^2z} - A_q \right)^{-1} \\
&= \frac{1 - h^4}{(1 - h^2z)^2} \left( \frac{z - h^2}{1 - h^2z} - A_q \right)^{-1} - \frac{h^2}{1 - h^2z}.
\end{align*}
  Choose any $ \phi \in \mathcal H_0^\infty(\Sigma_\theta) $.
  Invoking the integral identity
  \[
    \phi(\widetilde A_q) = \frac{1}{2\pi i}
    \int_{\partial\Sigma_\theta} \phi(z)(z - \widetilde A_q)^{-1} \, \mathrm{d}z,
  \]
  and using Cauchy's formula, which implies
  \[
    \frac1{2\pi i}\int_{\partial\Sigma_\theta} \phi(z) \frac{h^2}{1 - h^2z}
    \, \mathrm{d}z = -\phi(h^{-2}),
  \]
  we deduce that
  \begin{equation}
    \label{eq:731}
    \phi(\widetilde A_q) = \phi(h^{-2}) + \frac{1}{2\pi i}
    \int_{\partial\Sigma_\theta} \phi(z) \frac{1 - h^4}{(1 - h^2z)^2} \left( \frac{z - h^2}{1 - h^2z} - A_q \right)^{-1} \,
    \mathrm{d}z. 
  \end{equation}
  Since $ |\operatorname{Arg}((z-h^2)/(1-h^2z)) | > \theta $ for all $ z \in \partial\Sigma_{\theta} $,
  it is straightforward to verify that the integral in the above equation is indeed convergent
  by part (i) of \cref{lem:Aq}. Fix any $ 0 < \theta' < \theta $ and define
  \begin{equation*}
    \varphi(\lambda) := \phi\Big(\frac{{h^2}+\lambda}{1+{h^2}\lambda}\Big) - \phi\Big(\frac1{h^2}\Big),
    \quad \forall \lambda \in \overline{\Sigma_{\theta'}}.
  \end{equation*}
Given that \( \phi \) is analytic in \( \Sigma_\theta \) and uniformly
bounded therein, an elementary computation confirms that \( \varphi \) is analytic
on \( \overline{\Sigma_{\theta'}} \), and satisfies
\[
    \limsup_{r \to \infty} \sup_{\omega \in [-\theta', \theta']} r |\varphi(re^{i\omega})| < \infty.
\]
  For the second term in the right-hand side of \cref{eq:731}, 
  we perform the following analysis:
  \begin{align*} 
  & \frac1{2\pi i}\int_{\partial\Sigma_\theta} \phi(z)
  \frac{1-{h^4}}{(1-{h^2} z)^2}
  \left(
    \frac{z - {h^2}}{1 - {h^2} z} - A_q
  \right)^{-1} \, \mathrm{d}z \\
  \stackrel{\text{(i)}}{=} &
  \frac1{2\pi i}\int_{\partial\Sigma_\theta} \phi(z)
  \frac{1-{h^4}}{(1-{h^2} z)^2} \frac1{2\pi i}
  \int_{\partial\Sigma_{\theta'}}
  \left(
    \frac{z - {h^2}}{1 - {h^2} z} - \lambda
  \right)^{-1} (\lambda - A_q)^{-1} \, \mathrm{d}\lambda \, \mathrm{d}z \\
  \stackrel{\text{(ii)}}{=} &
  \frac1{2\pi i}\int_{\partial\Sigma_{\theta'}}  \frac1{2\pi i}
  \int_{\partial\Sigma_{\theta}}
  \phi(z)
  \frac{1-{h^4}}{(1-{h^2} z)^2}
  \left(
    \frac{z - {h^2}}{1 - {h^2} z} - \lambda
  \right)^{-1}  \, \mathrm{d}z \, (\lambda - A_q)^{-1} \, \mathrm{d}\lambda \\
  ={}&
  \frac1{2\pi i}\int_{\partial\Sigma_{\theta'}}  \frac1{2\pi i}
  \int_{\partial\Sigma_{\theta}}
  \phi(z) \left[
    \frac1{z - \frac{h^2+\lambda}{1+h^2\lambda}} -
    \frac1{z - \frac1{h^2}} 
  \right] \, \mathrm{d}z \, (\lambda - A_q)^{-1} \, \mathrm{d}\lambda,
\end{align*}
where in (i) we use the fact that $ (z-h^2)/(1-h^2z) \not\in \overline{\Sigma_{\theta'}} $,
and in (ii) we apply Fubini's theorem to interchange the order of integration.
Since $ (h^2+\lambda)/(1 + h^2\lambda) $ belongs to $ \Sigma_{\theta} $
for all $ \lambda \in \partial\Sigma_{\theta'} $,
using Cauchy's formula and the definition of $ \varphi $, we can further deduce that
\begin{align*} 
  & \frac1{2\pi i}\int_{\partial\Sigma_\theta} \phi(z)
  \frac{1-{h^4}}{(1-{h^2} z)^2}
  \left(
    \frac{z - {h^2}}{1 - {h^2} z} - A_q
  \right)^{-1} \, \mathrm{d}z \\
  ={} &
  \frac1{2\pi i} \int_{\partial\Sigma_{\theta'}}  \left[
    \phi\Big(\frac{h^2+\lambda}{1+h^2\lambda}\Big) -
    \phi\Big(\frac1{h^2}\Big)
  \right] \, (\lambda - A_q)^{-1} \, \mathrm{d}\lambda \\
  ={}&
  \frac1{2\pi i}\int_{\partial\Sigma_{\theta'}}
  \varphi(\lambda) (\lambda - A_q)^{-1} 
  \, \mathrm{d}\lambda \\
  ={} &
  \varphi(A_q).
\end{align*}
By combining the above result with the equality \cref{eq:731}, we obtain
\[
  \phi(\widetilde A_q) = \phi(h^{-2}) + \varphi(A_q).
\]
Using Corollary~3.1.11 from \cite{Pruss2016}, we further arrive at
\[
  \phi(\widetilde A_q) = \phi(h^{-2}) + \varphi(0)(1+A_q)^{-1} + \Psi(A_q),
\]
where the function $ \Psi $ is defined as
\[
  \Psi(\lambda) := \varphi(\lambda) - \frac{\varphi(0)}{1+\lambda},
  \quad \forall \lambda \in \Sigma_{\theta'}.
\]
This leads to the inequality
\[
  \norm{\phi(\widetilde A_q)}_{\mathcal L(X_0^q)}
  \leqslant \abs{\phi(h^{-2})} + \abs{\varphi(0)} \norm{(1+A_q)^{-1}}_{\mathcal L(X_0^q)} +
  \norm{\Psi(A_q)}_{\mathcal L(X_0^q)}.
\]
The desired inequality \cref{eq:wtAq-Hinf} then follows from \cref{lem:Aq},
alongside the observation that $ \Psi \in \mathcal H_0^\infty(\Sigma_{\theta'}) $ and
\[
  \abs{\phi(h^{-2})} + \abs{\varphi(0)} +
  \norm{\Psi}_{\mathcal H_0^\infty(\Sigma_{\theta'})}
  \leqslant C_\theta \norm{\phi}_{\mathcal H_0^\infty(\Sigma_{\theta})}.
\]

\textbf{Step 2.} For any $ \phi \in \mathcal H_0^\infty(\Sigma_\theta) $,
we aim to establish the decomposition
\begin{equation}
  \label{eq:918}
  \phi(\mathcal{A}_h) P_h - \phi(\widetilde A_q) =
  I_1(\phi) + I_2(\phi) + I_3(\phi),
\end{equation}
where 
\begin{align}
  I_1(\phi) &:= \frac1{2\pi i} \int_{\Gamma_1}
  \phi(z)\left[ (z-\mathcal{A}_h)^{-1}P_h - (z-\widetilde A_q)^{-1} \right] \, \mathrm{d}z,
  \label{eq:I1} \\
  I_2(\phi) &:= \frac1{2\pi i} \int_{\Gamma_2}
  \phi(z)\left[ (z-\mathcal{A}_h)^{-1}P_h - (z-\widetilde A_q)^{-1} \right] \, \mathrm{d}z,
  \label{eq:I2} \\
  I_3(\phi) &:= \frac1{2\pi i} \int_{\Gamma_3} \phi(z)\left[ (z-\mathcal{A}_h)^{-1}P_h - (z-\widetilde A_q)^{-1} \right] \, \mathrm{d}z.
  \label{eq:I3}
\end{align}
The contours $\Gamma_1$, $\Gamma_2$, and $\Gamma_3$ are parametrized by
\begin{align*}
  \Gamma_1 &: z = re^{i\theta}, \quad 0 \leqslant r \leqslant \frac{c^*}{h^2 \cos \theta},  \\
  \Gamma_2 &: z = re^{-i\theta}, \quad 0 \leqslant r \leqslant \frac{c^*}{h^2 \cos \theta}, \\
  \Gamma_3 &: z = \frac{c^*}{h^{2}} + i\xi, \quad -\frac{c^*}{h^{2}} \tan \theta \leqslant \xi \leqslant \frac{c^*}{h^{2}} \tan \theta,
\end{align*}
with orientations as illustrated in \Cref{fig:I1I2I3}. 
\begin{figure}[htbp]
    \centering
    \includegraphics[width=0.5\textwidth]{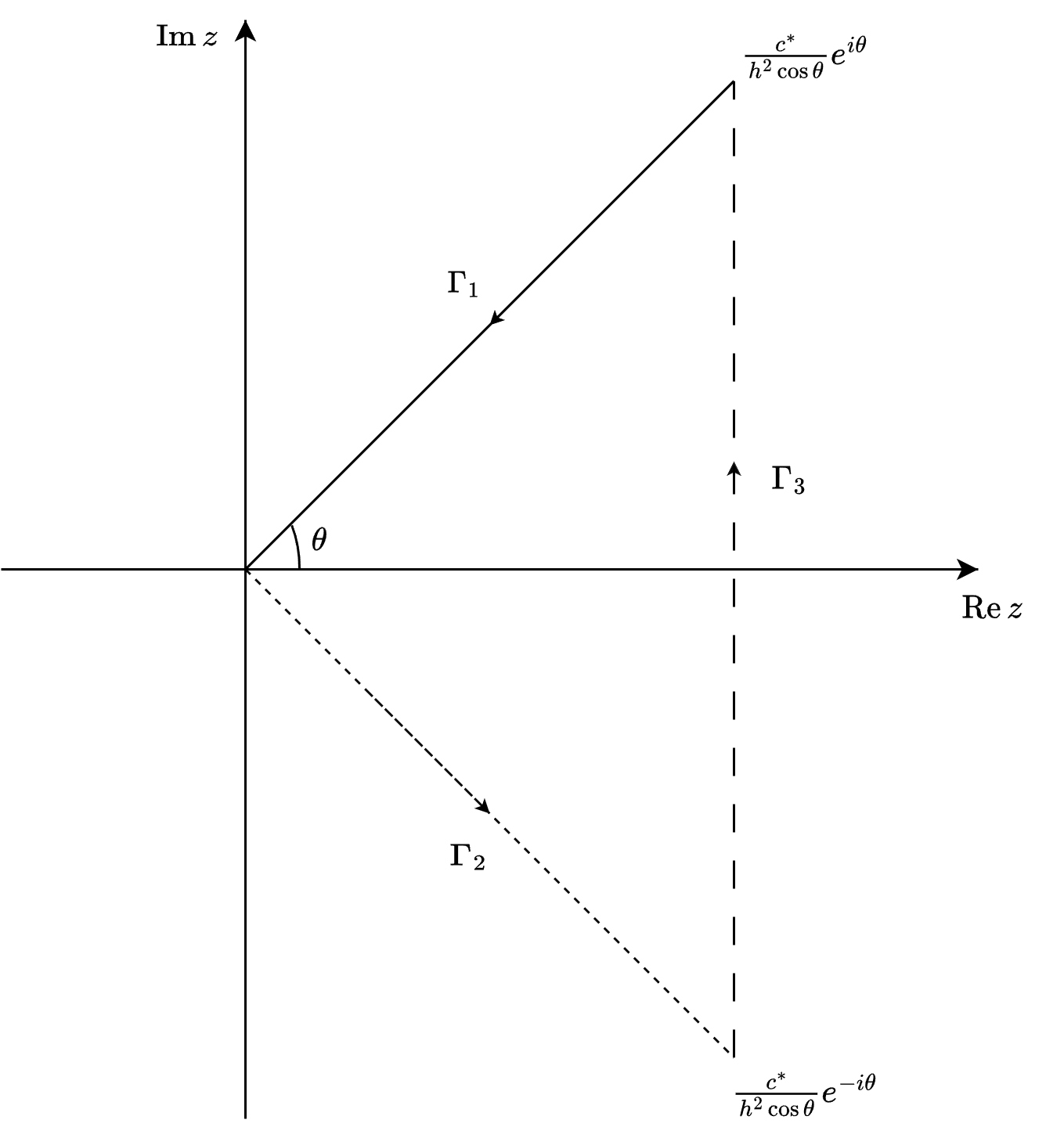}
    \caption{$\Gamma_1$, $\Gamma_2$, and $\Gamma_3$.}
    \label{fig:I1I2I3}
\end{figure}
Here, $ c^* $ is a positive constant depending only on $ q $, $ \varrho $, and the
geometric characteristics of $ \mathcal O $.
From the definition of $ \widetilde A_q $ in \cref{eq:wtAq}
and the properties of $ A_q $ stated in part (i) of \cref{lem:Aq}, we deduce that
\begin{align*}
  \norm{\widetilde A_q}_{\mathcal{L}(X_0^q)}
  &\leqslant h^2 \norm{(1 + h^2 A_q)^{-1}}_{\mathcal{L}(X_0^q)} +
  \norm{A_q(1+h^2A_q)^{-1}}_{\mathcal{L}(X_0^q)} \\
  &= \norm{(h^{-2} + A_q)^{-1}}_{\mathcal{L}(X_0^q)} +
  h^{-2}\norm{A_q(h^{-2}+A_q)^{-1}}_{\mathcal{L}(X_0^q)} \\
  &\leqslant C_{q,\mathcal{O}} h^{-2}.
\end{align*}
Moreover, \cref{lem:etAh}(ii) with $ \alpha = 1 $ gives
$ \norm{\mathcal{A}_h}_{\mathcal{L}(X_{0,h}^q)} \leqslant C_{q,\varrho,\mathcal{O}} h^{-2} $.
Consequently, there exists a positive constant \(c^*\), solely determined
by \(q\), \(\varrho\), and the geometric characteristics of \(\mathcal{O}\),
such that
\begin{equation}
  \label{eq:7000}
  c^* h^{-2} \geqslant 2 \max\left\{
    \norm{\mathcal{A}_h}_{\mathcal{L}(X_{0,h}^q)}, \,
    \norm{\widetilde A_q}_{\mathcal{L}(X_0^q)}
  \right\},
\end{equation}
yielding the resolvent set inclusion
\[
  \left\{z \in \mathbb{C} \, \middle| \,
  \operatorname{Re} z \geqslant c^* h^{-2} \right\}
  \subseteq \rho(\widetilde A_q) \cap \rho(\mathcal{A}_h).
\]
Given that 
\[
  \mathbb{C} \setminus (0,\infty) \subseteq \rho(\widetilde A_q) \cap \rho(\mathcal{A}_h),
\]
we can refine this inclusion to yield
\begin{equation}
  \label{eq:786}
  \mathbb{C} \setminus (0,c^*h^{-2})
  \subseteq \rho(\widetilde A_q) \cap \rho(\mathcal{A}_h).
\end{equation}
Therefore, for any \( \phi \in \mathcal{H}_0^\infty(\Sigma_\theta) \), utilizing the identity
\begin{align*}
  \phi(\mathcal{A}_h) P_h - \phi(\widetilde A_q) =
  \frac1{2\pi i} \int_{\partial\Sigma_\theta}
  \phi(z)\left[ (z-\mathcal{A}_h)^{-1}P_h - (z-\widetilde A_q)^{-1} \right]
  \, \mathrm{d}z,
\end{align*}
we can invoke Cauchy's integral theorem to deform the contour of integration,
thereby achieving the desired decomposition \cref{eq:918}.

\textbf{Step 3.} We proceed to prove that
\begin{equation}
  \label{eq:I1+I2}
  \norm{I_1(\phi) + I_2(\phi)}_{\mathcal L(X_0^q)}
  \leqslant C_{\theta,q,\varrho,\mathcal O}
  \norm{\phi}_{\mathcal H_0^\infty(\Sigma_\theta)}
  \quad\text{for all } \phi \in \mathcal H_0^\infty(\Sigma_\theta).
\end{equation}
To this end, consider an arbitrary element $\phi$ within $\mathcal{H}_0^\infty(\Sigma_\theta)$.
For any $z$ on the boundary $\partial\Sigma_{\theta}$, the following identity holds:
\begin{align*}
(z-A_q)^{-1} - (z-\widetilde A_q)^{-1} &= (A_q - \widetilde A_q)(z-\widetilde A_q)^{-1}(z-A_q)^{-1} \\
&= \frac{h^2}{1-h^2z} (A_q^2 - 1) \left( \frac{z-h^2}{1-h^2z} - A_q \right)^{-1} (z-A_q)^{-1},
\end{align*}
where the last step uses the definition of $ \widetilde A_q $ in \cref{eq:wtAq}.
Given that $\frac{z-h^2}{1-h^2z}$ lies exterior to $\Sigma_\theta$ for
all $z \in \partial\Sigma_\theta$, using the properties of $ A_q $ in part (i) of \cref{lem:Aq},
we confirm the uniform boundedness in $ \mathcal L(X_0^q) $ of the operators
\[
A_q^2\left( \frac{z-h^2}{1-h^2z} - A_q \right)^{-1} (z-A_q)^{-1}
\]
and
\[
\left( \frac{z-h^2}{1-h^2z} - A_q \right)^{-1} (z-A_q)^{-1}
\]
with respect to both $h$ and $z \in \partial\Sigma_\theta $.
This leads to the conclusion that, for any $z \in \partial\Sigma_\theta$,
\[
\norm{(z-A_q)^{-1} - (z-\widetilde A_q)^{-1}}_{\mathcal{L}(X_0^q)} \leqslant C_{\theta,q,\varrho,\mathcal{O}} \frac{h^2}{|1-h^2z|},
\]
which, combined with the inequality
\[
\frac{1}{|1-h^2z|} \leqslant \frac{1}{\sqrt{(1-\cos\theta)(1+h^4|z|^2)}}
\leqslant \frac{1}{\sqrt{1-\cos\theta}}, \quad \forall z \in \partial\Sigma_\theta,
\]
yields
\[
\norm{(z-A_q)^{-1} - (z-\widetilde A_q)^{-1}}_{\mathcal{L}(X_0^q)} \leqslant C_{\theta,q,\varrho,\mathcal{O}} h^2.
\]
Combining this result with the estimate from \cref{lem:z-Ah}, we attain
\[
  \norm{(z-\mathcal{A}_h)^{-1} P_h - (z-\widetilde A_q)^{-1}}_{\mathcal{L}(X_0^q)} 
  \leqslant{}
  C_{\theta,q,\varrho,\mathcal{O}} h^2, \quad \forall z \in \partial\Sigma_{\theta}.
\]
Using this inequality, together with \cref{eq:I1,eq:I2}, we deduce that
\begin{align*}
 &\norm{I_1(\phi) + I_2(\phi)}_{\mathcal{L}(X_0^q)} \\
 \leqslant{}
 & C_{\theta,q,\varrho,\mathcal{O}}
 \norm{\phi}_{\mathcal{H}_0^\infty(\Sigma_\theta)}
 h^{-2} \max_{z \in \partial\Sigma_\theta}
 \norm{(z-\mathcal{A}_h)^{-1}P_h - (z-\widetilde A_q)^{-1}}_{\mathcal{L}(X_0^q)} \\
 \leqslant{}
 & C_{\theta,q,\varrho,\mathcal O} \norm{\phi}_{\mathcal H_0^\infty(\Sigma_\theta)}.
\end{align*}
This confirms the claim \cref{eq:I1+I2}.

\textbf{Step 4.} Let us prove that 
\begin{equation}
  \label{eq:I3-bound}
  \norm{I_3(\phi)}_{\mathcal L(X_0^q)}
  \leqslant C_{\theta,q,\varrho,\mathcal O}
  \norm{\phi}_{\mathcal H_0^\infty(\Sigma_\theta)}
  \quad\text{for all } \phi \in \mathcal H_0^\infty(\Sigma_\theta).
\end{equation}
For any $z$ in the complex segment $c^*h^{-2} + i(-c^*h^{-2} \tan\theta, c^*h^{-2} \tan\theta)$,
we begin by bounding the norm as follows:
\begin{align*}
    & \norm{(z-\mathcal{A}_h)^{-1}P_h - (z-\widetilde A_q)^{-1}}_{\mathcal{L}(X_0^q)} \\
    \leqslant{} 
    & \norm{(z-\mathcal{A}_h)^{-1}P_h}_{\mathcal{L}(X_0^q)} + 
    \norm{(z-\widetilde A_q)^{-1}}_{\mathcal{L}(X_0^q)} \\
    \leqslant{}
    & C_{q,\varrho} \norm{(z-\mathcal{A}_h)^{-1}}_{\mathcal{L}(X_{0,h}^q)} + 
    \norm{(z-\widetilde A_q)^{-1}}_{\mathcal{L}(X_0^q)}
    \quad\text{(by \cref{lem:Ph}(i))} \\
    \leqslant{}
    & \frac{C_{q,\varrho}}{|z| - \norm{\mathcal{A}_h}_{\mathcal{L}(X_{0,h}^q)}} + 
    \frac{1}{|z| - \norm{\widetilde A_q}_{\mathcal{L}(X_0^q)}} \\
    \leqslant{} & C_{\theta,q,\varrho,\mathcal{O}} h^2,
\end{align*}
where the third and fourth inequalities follow from \eqref{eq:7000}, the fact that $|z| \geqslant c^*h^{-2}$,
and \cite[Equation~(1.11)]{Yagi2010}.
Here, we recall that $ c^* $ is a positive constant that depends solely on $ q $, $ \varrho $, and the geometric characteristics of $ \mathcal O $.
Consequently, given any $\phi \in \mathcal{H}_0^\infty(\Sigma_\theta)$,
from \cref{eq:I3} we deduce that
\begin{align*}
  \norm{I_3(\phi)}_{\mathcal{L}(X_0^q)} 
    &\leqslant 
    C_{\theta,q,\varrho,\mathcal{O}} \norm{\phi}_{\mathcal{H}_0^\infty(\Sigma_\theta)}
    h^{-2} \times \\
    & \qquad \max_{z \in \frac{c^*}{h^2} + i\left(-\frac{c^*}{h^2}\tan\theta,\frac{c^*}{h^2}\tan\theta\right)} 
    \norm{(z-\mathcal{A}_h)^{-1}P_h - (z - \widetilde A_q)^{-1}}_{\mathcal{L}(X_0^q)} \\
    &\leqslant
    C_{\theta,q,\varrho,\mathcal{O}} \norm{\phi}_{\mathcal{H}_0^\infty(\Sigma_\theta)}.
\end{align*}
This completes the proof of \eqref{eq:I3-bound}.

\textbf{Step 5.}
For \(\phi \in \mathcal{H}_0^\infty(\Sigma_\theta)\),
we combine the estimates from \cref{eq:918,eq:I1+I2,eq:I3-bound} to establish  
\[
  \|\phi(\mathcal{A}_h)P_h - \phi(\widetilde{A}_q)\|_{\mathcal{L}(X_0^q)} 
  \leqslant C_{\theta,q,\varrho,\mathcal{O}} \|\phi\|_{\mathcal{H}_0^\infty(\Sigma_\theta)}.
\]  
In conjunction with the bound \eqref{eq:wtAq-Hinf}, this implies  
\[
  \|\phi(\mathcal{A}_h)P_h\|_{\mathcal{L}(X_0^q)} 
  \leqslant C_{\theta,q,\varrho,\mathcal{O}} \|\phi\|_{\mathcal{H}_0^\infty(\Sigma_\theta)}.
\]  
The identity \(P_hu_h = u_h\) for \(u_h \in X_{0,h}^q\) then yields the required bound \cref{eq:Ah-Hinf}.
This completes the proof of assertion (i).

  \textbf{Step 6.}
To prove assertion (ii), let $ t \in \mathbb{R} $ be fixed. For $ \epsilon \in (0,1) $, consider the holomorphic function defined on $ \Sigma_\theta $ as  
$$
\phi_{\epsilon,t}(z) := z^{it} \frac{z^\epsilon}{(1 + z)^{2\epsilon}},
$$  
where $ z^{it} $ uses the principal branch of the logarithm with arguments in $ (-\pi, \pi) $.  
By \cref{eq:786} and Cauchy's theorem, we deform the integration contour to obtain  
$$
\mathcal{A}_h^{it} = \frac{1}{2\pi i} \int_{\Gamma} z^{it}(z - \mathcal{A}_h)^{-1} \, \mathrm{d}z \quad \text{and} \quad \phi_{\epsilon,t}(\mathcal{A}_h) = \frac{1}{2\pi i} \int_{\Gamma} \phi_{\epsilon,t}(z)(z - \mathcal{A}_h)^{-1} \, \mathrm{d}z,
$$  
where $ \Gamma := \Gamma_1 \cup \Gamma_2 \cup \Gamma_3 $ and the orientation of $ \Gamma $ is consistent with those of $ \Gamma_1 $, $ \Gamma_2 $, and $ \Gamma_3 $; see \cref{fig:I1I2I3}. The uniform bound  
$$
\sup_{\epsilon \in (0,1)} \norm{\phi_{\epsilon,t}(\cdot)(\cdot - \mathcal{A}_h)^{-1}}_{L^\infty(\Gamma; \mathcal{L}(X_{0,h}^q))} < \infty,
$$  
together with pointwise convergence $ \lim_{\epsilon\to 0} \phi_{\epsilon,t}(z)(z - \mathcal{A}_h)^{-1} = z^{it}(z - \mathcal{A}_h)^{-1} $ on $ \Gamma \setminus \{0\} $, allows the application of Lebesgue's dominated convergence theorem to conclude  
$$
\lim_{\epsilon \to 0^+} \phi_{\epsilon,t}(\mathcal{A}_h) = \mathcal{A}_h^{it} \quad \text{in } \mathcal{L}(X_{0,h}^q).
$$  
Since $ \phi_{\epsilon,t} \in \mathcal{H}_0^\infty(\Sigma_\theta) $, estimate \eqref{eq:Ah-Hinf} yields  
$$
\norm{\phi_{\epsilon,t}(\mathcal{A}_h)}_{\mathcal{L}(X_{0,h}^q)} \leqslant C_{\theta,q,\varrho,\mathcal{O}} \sup_{z \in \Sigma_{\theta}} \Big| z^{it} \frac{z^{\epsilon}}{(1+z)^{2\epsilon}} \Big|
\leqslant C_{\theta,q,\varrho,\mathcal{O}} e^{\theta |t|}.
$$  
Thus, the desired inequality \cref{eq:Ah-BIP} follows. This completes the proof of assertion (ii).

  \textbf{Step 7.} Let us establish assertion (iii). Fix any \( \theta' \in (0,\theta) \).  
For any \( z \in \mathbb{C} \setminus \overline{\Sigma_\theta} \), define  
\[
\phi_z(\lambda) := \frac{z}{z-\lambda}, \quad \lambda \in \Sigma_{\theta'}.
\]  
It is straightforward to verify that \( \phi_z(\cdot): \Sigma_{\theta'} \to \mathbb{C} \) is analytic for all \( z \in \mathbb{C} \setminus \overline{\Sigma_\theta} \), and that  
\[
\sup_{z \in \mathbb{C} \setminus \overline{\Sigma_\theta}} \sup_{\lambda \in \Sigma_{\theta'}} |\phi_z(\lambda)| < \infty.
\]  
By assertion (i) of this theorem, we have  
\[
\norm{\phi(\mathcal{A}_h)}_{\mathcal{L}(X^q_{0,h})} \leqslant C_{\theta',q,\varrho,\mathcal{O}} \norm{\phi}_{\mathcal{H}_0^\infty(\Sigma_{\theta'})}, \quad \forall \phi \in \mathcal{H}_0^\infty(\Sigma_{\theta'}).
\]  
Furthermore, since \( X_{0,h}^q \) is a closed subspace of \( X_0^q \), it follows from Propositions 7.5.3 and 7.5.4 in \cite{HytonenWeis2017} that \( X_{0,h}^q \) inherits Pisier's contraction property, with the associated constants being independent of \( h \).  
Consequently, by Theorem 10.3.4 from \cite{HytonenWeis2017}, we deduce that the family  
\[
\bigl\{ \phi_z(\mathcal{A}_h) \mid z \in \mathbb{C} \setminus \overline{\Sigma_\theta} \bigr\}
\]  
is \(\mathcal{R}\)-bounded on \( X_{0,h}^q \), and its \(\mathcal{R}\)-boundedness constant is uniformly bounded with respect to \( h \).  
This establishes assertion (iii) of \cref{thm:Ah-Hinf}, thereby completing the proof of \cref{thm:Ah-Hinf}.

  \hfill$\blacksquare$\par

\subsection{Proof of \texorpdfstring{\cref{thm:DSMLP}}{}}
Having established that the boundedness of the $ H^\infty $-calculus for $ \mathcal{A}_h $ is independent of the spatial mesh size $ h $, as shown in \cref{thm:Ah-Hinf}, the discrete stochastic maximal $ L^p $-regularity estimate \cref{eq:DSMLP} follows directly from \cite[Theorem~3.5]{Neerven2012b}. For completeness, we provide a concise outline.

For any $ r > 0 $ and $ g_h \in L_{\mathbb{F}}^p(\Omega \times \mathbb{R}_{+}; \gamma(H, X_{0,h}^q)) $, define the operator
$$
(\mathcal{J}_h(r)g_h)(t) := \frac{1}{\sqrt{r}} \int_{\max(0,t-r)}^t g_h(s) \, \mathrm{d}W_H(s), \quad t \geqslant 0.
$$
Using \cite[Equation~(2.5)]{Neerven2012b} and Minkowski’s inequality, one verifies that
$$
\mathcal{J}_h(r) \in \mathcal{L}\big( L_{\mathbb{F}}^p(\Omega \times \mathbb{R}_{+}; \gamma(H, X_{0,h}^q)), L^p(\Omega \times \mathbb{R}_{+}; X_{0,h}^q) \big) \quad \text{for all } r > 0.
$$
Let $ \mathscr{I} := \{ \mathcal{J}_h(r) \mid r > 0 \} $ denote this family.
Since $ X_{0,h}^q $ is a closed subspace of $ L^q(\mathcal{O}) $, and the $ \gamma $-Fubini isomorphism theorem \cite[Theorem~9.4.8]{HytonenWeis2017} ensures the isomorphism $ \gamma(H,L^q(\mathcal{O})) \simeq L^q(\mathcal{O}; H) $, it follows from \cite[Theorem~3.1]{Neerven2012} that $ \mathscr{I} $ is $ \mathcal{R} $-bounded, with its $ \mathcal{R} $-boundedness constant independent of $ h $.

Next, recall from \cite[Proposition~7.1.4]{HytonenWeis2017} that the space $ X_0^q $ has type $ 2 $. Since $ X_{0,h}^q $ is a closed subspace of $ X_0^q $, it inherits this property, and the associated type $ 2 $ constant for $ X_{0,h}^q $ remains uniformly bounded with respect to $ h $; see \cite[p.~54]{HytonenWeis2017} for definitions.

Furthermore, by \cref{thm:Ah-Hinf}, for every $ \theta \in (0, \pi/2) $, the boundedness constant of the $ H^\infty(\Sigma_\theta) $-calculus for $ \mathcal{A}_h $ is independent of $ h $. It is standard that there exists $ r_0 > 0 $, independent of $ h $, such that $ \{ \lambda \in \mathbb{C} \mid |\lambda| \leqslant r_0 \} $ lies in the resolvent set of $ \mathcal{A}_h $.

Combining these results, the discrete stochastic maximal $ L^p $-regularity estimate \cref{eq:DSMLP} follows from \cite[Theorem~3.5]{Neerven2012b}. This completes the proof.

\subsection{Proof of \texorpdfstring{\cref{thm:conv}}{}}
\label{ssec:2} 
We begin with several technical estimates, which are stated in Lemmas~\ref{lem:Aq-Ah-alpha}--\ref{lem:key} below.
\begin{lemma}
  \label{lem:Aq-Ah-alpha}
  For any $ 0 < \alpha \leqslant 1 $ and $ q \in (1,\infty) $, we have
  \begin{equation}
    \label{eq:Aq-Ah-alpha}
    \norm{\mathcal{A}_h^{-\alpha}P_h - A_q^{-\alpha}}_{\mathcal L(X_0^q)} \leqslant
    C_{\alpha,q,\varrho,\mathcal O} h^{2\alpha}.
  \end{equation}
\end{lemma}
\begin{proof}
  The proof relies on \cref{lem:Ph,lem:z-Ah-inv,lem:Aq,lem:z-Ah}. 
  The case $\alpha = 1$ follows directly from the identity
  $\mathcal{A}_h^{-1} P_h - A_q^{-1} = (R_h-I)A_q^{-1}$, inequality \eqref{eq:Rh-conv-1}, and the equality
  $\norm{A_q^{-1}}_{\mathcal L(X_0^q,X_1^q)} = 1$.
  For $0 < \alpha < 1$, set $\theta = \pi/4$.
  Using the Dunford integral representation 
  \begin{align*}
    \mathcal{A}_h^{-\alpha}P_h - A_q^{-\alpha} &= \frac{1}{2\pi i}
    \int_{\partial\Sigma_\theta} \lambda^{-\alpha}
    \bigl[ (\lambda - \mathcal{A}_h)^{-1}P_h - (\lambda - A_q)^{-1} \bigr]
    \, \mathrm{d}\lambda,
  \end{align*}
  we estimate the norm as follows:
  \begin{align*}
    & \norm{\mathcal{A}_h^{-\alpha}P_h - A_q^{-\alpha}}_{\mathcal L(X_0^q)} \\
    \leqslant{} &
    \frac1{2\pi} \int_0^\infty 
    r^{-\alpha} \norm{(re^{i\theta}-\mathcal{A}_h)^{-1}P_h - (re^{i\theta}-A_q)^{-1}}_{\mathcal L(X_0^q)}
    \, \mathrm{d}r \\
    {} &
    \quad {}+ \frac1{2\pi} \int_0^\infty 
    r^{-\alpha} \norm{(re^{-i\theta}-\mathcal{A}_h)^{-1}P_h - (re^{-i\theta}-A_q)^{-1}}_{\mathcal L(X_0^q)}
    \, \mathrm{d}r.
  \end{align*}
  For any \(r \in (0, \infty)\), we have 
  \begin{align*}
  & \|(re^{\pm i\theta} - \mathcal{A}_h)^{-1} P_h - (re^{\pm i\theta} - A_q)^{-1}\|_{\mathcal{L}(X_0^q)} \\
  \leqslant{}
  & \|(re^{\pm i\theta} - \mathcal{A}_h)^{-1} P_h \|_{\mathcal{L}(X_0^q)} 
  + \|(re^{\pm i\theta} - A_q)^{-1}\|_{\mathcal{L}(X_0^q)} \\
  \leqslant{}
  &\frac{C_{q,\varrho,\mathcal{O}}}{1 + r},
  \end{align*}
  by Lemma \ref{lem:Ph}(i), inequality \eqref{eq:z-Ah-inv}, and Lemma \ref{lem:Aq}(i).
  Additionally, by \cref{lem:z-Ah}, we obtain
  \[
  \|(re^{\pm i\theta} - \mathcal{A}_h)^{-1} P_h - (re^{\pm i\theta} - A_q)^{-1}\|_{\mathcal{L}(X_0^q)} \leqslant C_{q,\varrho,\mathcal{O}} h^2.
  \]
  Combining these estimates, we get
  \begin{align*}
  \|\mathcal{A}_h^{-\alpha} P_h - A_q^{-\alpha}\|_{\mathcal{L}(X_0^q)}
  & \leqslant C_{q,\varrho,\mathcal{O}} \int_0^\infty r^{-\alpha}
  \min\{h^2, r^{-1}\} \, \mathrm{d}r  \\
  &= C_{q,\varrho,\mathcal O} \Big(
    \int_0^{h^{-2}} h^2r^{-\alpha} \, \mathrm{d}r
    + \int_{h^{-2}}^\infty r^{-\alpha-1} \, \mathrm{d}r
    \Big) \\
  & \leqslant C_{\alpha,q,\varrho,\mathcal{O}} h^{2\alpha}.
  \end{align*}
  This completes the proof of \eqref{eq:Aq-Ah-alpha} for \(0 < \alpha < 1\).
\end{proof}

  \begin{lemma}
    \label{lem:lxy}
    Let $ \alpha \in (0,1] $ and $ q \in (1,\infty) $. Then,
    for any $ v_h \in X_h $, the following inequality holds:
    \[
      \norm{v_h}_{X_{-\alpha}^q} \leqslant
      C_{\alpha,q,\varrho,\mathcal O}
      \norm{v_h}_{X_{-\alpha,h}^q}.
    \]
  \end{lemma}
  \begin{proof}
    We proceed as follows:
    \begin{align*}
      \norm{v_h}_{X_{-\alpha}^q} 
      &= \norm{A_q^{-\alpha}v_h}_{X_0^q} \\
      &\leqslant
      \norm{A_q^{-\alpha}v_h - \mathcal{A}_h^{-\alpha}v_h}_{X_0^q} +
      \norm{\mathcal{A}_h^{-\alpha}v_h}_{X_0^q} \\
      &\overset{\text{(i)}}{\leqslant}
      C_{\alpha,q,\varrho,\mathcal{O}} h^{2\alpha} \norm{v_h}_{X_0^q} +
      \norm{\mathcal{A}_h^{-\alpha}v_h}_{X_{0,h}^q} \\
      &\overset{\text{(ii)}}{\leqslant}
      C_{\alpha,q,\varrho,\mathcal{O}} \norm{\mathcal{A}_h^{-\alpha}v_h}_{X_{0,h}^q} \\
      &=
      C_{\alpha,q,\varrho,\mathcal{O}} \norm{v_h}_{X_{-\alpha,h}^q},
    \end{align*}
    where step (i) applies Lemma \ref{lem:Aq-Ah-alpha},
    and step (ii) utilizes Lemma \ref{lem:etAh}(ii).
    This establishes the desired inequality, thereby concluding the proof.
  \end{proof}

  \begin{lemma}
    \label{lem:key}
    Let $ p,q \in (1,\infty) $. Then, for any $ v_h \in X_h $, the following inequality holds:
    \[
      \norm{v_h}_{X_{1-1/p,h}^q} \leqslant
      C_{p,q,\varrho,\mathcal O}
      (1 + \abs{\ln h}^{1-1/p}) \norm{v_h}_{(X_{0,h}^q,X_{1,h}^q)_{1-1/p,p}}.
    \]
  \end{lemma}
  \begin{proof}
    Starting from the identity  $ v_h = \int_{\mathbb{R}_+} \mathcal{A}_h e^{-t\mathcal{A}_h} v_h \, \mathrm{d}t $,  
    we deduce  
    $$
    \begin{aligned}
      \|v_h\|_{X_{1-1/p,h}^q}
      &\leqslant \int_{\mathbb{R}_+} \|\mathcal{A}_h^{2-1/p} e^{-t\mathcal{A}_h} v_h\|_{X_{0, h}^q} \, \mathrm{d}t \\
      &\leqslant \int_{0}^{2} \|\mathcal{A}_h^{2-1/p} e^{-t\mathcal{A}_h} v_h\|_{X_{0, h}^q} \, \mathrm{d}t
      + \int_{2}^{\infty} \|\mathcal{A}_h^{2-1/p} e^{-t\mathcal{A}_h}\|_{\mathcal{L}(X_{0,h}^q)}
      \|v_h\|_{X_{0, h}^q} \, \mathrm{d}t.
    \end{aligned}
    $$  
    By \cref{lem:etAh}(i), for all $ t > 0 $,  
    $$
    \begin{aligned}
      \|\mathcal{A}_h^{2-1/p}e^{-t\mathcal{A}_h}\|_{\mathcal{L}(X_{0,h}^q)}
      &= \|\mathcal{A}_h^{1-1/(2p)}e^{-\frac{t}{2}\mathcal{A}_h} \mathcal{A}_h^{1-1/(2p)}e^{-\frac{t}{2}\mathcal{A}_h}\|_{\mathcal{L}(X_{0,h}^q)} \\
      &\leqslant \|\mathcal{A}_h^{1-1/(2p)}e^{-\frac{t}{2}\mathcal{A}_h}\|_{\mathcal{L}(X_{0,h}^q)}^2 \\
      &\leqslant C_{p,q,\varrho,\mathcal{O}} \, t^{\frac{1}{p}-2},
    \end{aligned}
    $$  
    yielding  
    $$
    \int_{2}^{\infty} \|\mathcal{A}_h^{2-1/p}e^{-t\mathcal{A}_h}\|_{\mathcal{L}(X_{0,h}^q)}
    \|v_h\|_{X_{0,h}^q} \, \mathrm{d}t
    \leqslant C_{p,q,\varrho,\mathcal{O}} \|v_h\|_{X_{0,h}^q}.
    $$  
    Thus,  
    $$
    \|v_h\|_{X_{1-1/p,h}^q}
    \leqslant \int_{0}^{2} \|\mathcal{A}_h^{2-1/p} e^{-t\mathcal{A}_h} v_h\|_{X_{0, h}^q} \, \mathrm{d}t
    + C_{p,q,\varrho,\mathcal{O}} \|v_h\|_{X_{0,h}^q}.
    $$  
    The $h$-independent embedding $ (X_{0,h}^q, X_{1,h}^q)_{1-1/p,p} \hookrightarrow X_{0,h}^q $ (see \cref{rem:X0hq-X1hq})
    implies  
    $$
    \|v_h\|_{X_{1-1/p,h}^q}
    \leqslant \int_{0}^{2} \|\mathcal{A}_h^{2-1/p} e^{-t\mathcal{A}_h} v_h\|_{X_{0, h}^q} \, \mathrm{d}t
    + C_{p,q,\varrho,\mathcal{O}} \|v_h\|_{(X_{0,h}^q,X_{1,h}^q)_{1-1/p,p}}.
    $$  
    Therefore, it remains to prove  
    \begin{equation}
    \label{eq:tmp-task}
    \int_{0}^{2} \|\mathcal{A}_h^{2-1/p} e^{-t\mathcal{A}_h} v_h\|_{X_{0, h}^q} \, \mathrm{d}t
    \leqslant C_{p,q,\varrho,\mathcal{O}} \big(1 + |\ln h|^{1-1/p}\big)
    \|v_h\|_{(X_{0,h}^q,X_{1,h}^q)_{1-1/p,p}}.
    \end{equation}

    To this end, we proceed as follows:
    \begin{align*}
      & \int_0^{2} \normb{\mathcal{A}_h^{2-1/p}e^{-t\mathcal{A}_h}v_h}_{X_{0,h}^q} \, \mathrm{d}t \leqslant
      \int_{0}^{2} \normb{\mathcal{A}_h^{1-1/p} e^{-\frac{t}2\mathcal{A}_h}}_{\mathcal{L}(X_{0, h}^q)}
      \normb{\mathcal{A}_h e^{-\frac{t}2\mathcal{A}_h} v_h}_{X_{0, h}^q} \, \mathrm{d}t \\
      \leqslant{} 
      & \left(
        \int_{0}^{2} \normb{\mathcal{A}_h^{1-1/p} e^{-\frac{t}2\mathcal{A}_h}}_{\mathcal{L}(X_{0, h}^q)}^{\frac{p}{p-1}} \, \mathrm{d}t
      \right)^{1-\frac{1}{p}} \left(
        \int_{0}^{2} \normb{\mathcal{A}_h e^{-\frac{t}2\mathcal{A}_h} v_h}_{X_{0, h}^q}^p \, \mathrm{d}t
      \right)^{\frac{1}{p}},
    \end{align*}
    where the last inequality follows from Hölder's inequality.  
    For the first factor, we deduce that
    \begin{align*}
       \left( \int_0^{2} \norm{\mathcal{A}_h^{1-1/p} e^{-\frac{t}2\mathcal{A}_h}}_{\mathcal{L}(X_{0,h}^q)}^{\frac{p}{p-1}} \, \mathrm{d}t \right)^{1-\frac{1}{p}} 
      &\leqslant
       \left( \int_0^{h^2} \norm{\mathcal{A}_h^{1-1/p}}_{\mathcal L(X_{0,h}^q)}^{\frac{p}{p-1}}
      \norm{e^{-\frac{t}2\mathcal{A}_h}}_{\mathcal{L}(X_{0,h}^q)}^{\frac{p}{p-1}} \, \mathrm{d}t \right)^{1-\frac{1}{p}} \\
      & \quad {}+\left( \int_{\min(h^2,2)}^{2} \norm{\mathcal{A}_h^{1-1/p}e^{-\frac{t}2\mathcal{A}_h}}_{\mathcal L(X_{0,h}^q)}^{\frac{p}{p-1}} \, \mathrm{d}t \right)^{1-\frac{1}{p}} \\
      &\leqslant C_{p,q,\varrho,\mathcal{O}} \left[  \int_0^{h^2} h^{-2} \, \mathrm{d}t  +
      \int_{\min(h^2,2)}^{2} t^{-1} \mathrm{d}t  \right]^{1-\frac1p} \\
      &\leqslant C_{p,q,\varrho,\mathcal{O}} \big(1 + |\ln h|^{1-1/p}\big),
    \end{align*}
    where the second inequality uses \cref{lem:etAh}(ii) with $ \alpha = 1-1/p $ and \cref{lem:etAh}(i) with $ \alpha = 1-1/p $.
    For the second factor, by \cref{lem:etAh}(i), we may apply Proposition 6.2 of \cite{Lunardi2018} to obtain
    \begin{align*}
      \left( \int_0^{2} \normb{\mathcal{A}_h e^{-\frac{t}2\mathcal{A}_h} v_h}_{X_{0, h}^q}^p \, \mathrm{d}t \right)^{\frac{1}{p}}
      &= \left( 2\int_0^{1} \normb{\mathcal{A}_h e^{-t\mathcal{A}_h} v_h}_{X_{0, h}^q}^p \, \mathrm{d}t \right)^{\frac{1}{p}} \\
      & \leqslant C_{p, q, \varrho, \mathcal{O}} \norm{v_h}_{(X_{0, h}^q, X_{1, h}^q)_{1-1/p, p}}.
    \end{align*}
    Combining these estimates establishes \eqref{eq:tmp-task}. This completes the proof.
  \end{proof}

\begin{remark}
  \label{rem:X0hq-X1hq}
  Let $\alpha \in (0,1)$ and $p, q \in (1, \infty)$.
  For any $v_h \in X_h$, define $v := A_q^{-1} \mathcal{A}_h v_h$.
  Applying inequality~\eqref{eq:Aq-Ah-alpha} with $\alpha = 1$, we obtain
  \[
    \|v - v_h\|_{X_0^q} \leqslant C_{q,\varrho,\mathcal{O}} h^2 \|\mathcal{A}_h v_h\|_{X_0^q}.
  \]
  By the triangle inequality, the norm of $v_h$ satisfies
  \[
    \|v_h\|_{X_{0,h}^q} \leqslant \|v\|_{X_0^q} + C_{q,\varrho,\mathcal{O}} h^2 \|\mathcal{A}_h v_h\|_{X_0^q}.
  \]
  Using the relation $\|v\|_{X_1^q} = \|\mathcal{A}_h v_h\|_{X_0^q}$ and the continuous embedding $X_1^q \hookrightarrow X_0^q$, we deduce that
  \[
    \|v_h\|_{X_{0,h}^q} \leqslant C_{q,\varrho,\mathcal{O}} \|\mathcal{A}_h v_h\|_{X_0^q} = C_{q,\varrho,\mathcal{O}} \|v_h\|_{X_{1,h}^q}, \quad \forall v_h \in X_h.
  \]
  Consequently, invoking \cite[Proposition~1.3]{Lunardi2018}, we establish the inequality
  \[
    \|v_h\|_{X_{0,h}^q} \leqslant C_{\alpha,p,q,\varrho,\mathcal{O}} \|v_h\|_{(X_{0,h}^q, X_{1,h}^q)_{\alpha,p}}, \quad \forall v_h \in X_h.
  \]
\end{remark}

\begin{remark}
  We observe that the operator $ \mathcal{A}_h $ is symmetric and positive definite on the space $ X_{0,h}^2 $,
  and its smallest eigenvalue is uniformly bounded away from zero for all $h$.
  By \cite[Theorem~4.36]{Lunardi2018}, we are able to eliminate the logarithmic factor that is present in \cref{lem:key}
  for the specific case where $ p = q = 2 $. This result is a well-established one in the existing literature.
\end{remark}

Now we are ready to prove \cref{thm:conv}. The argument proceeds in three steps.

\textbf{Step 1 (Estimate for $y_h - P_hy$).}
From \cref{eq:y}, we have almost surely that
$$
y(t) + \int_0^t A_q y(s) \, \mathrm{d}s = \int_0^t f(s) \, \mathrm{d}W_H(s) \quad \text{in } X_{-1/2}^q \text{ for all } t \in [0,\infty).
$$
Applying the projection $ P_h $ to both sides yields
$$
\mathrm{d}P_h y(t) + P_h A_q y(t) \, \mathrm{d}t = P_h f(t) \, \mathrm{d}W_H(t), \quad t \geqslant 0.
$$
Subtracting this from \cref{eq:yh}, we obtain, almost surely,
$$
\mathrm{d}e_h(t) + \mathcal{A}_h e_h(t) \, \mathrm{d}t = (P_h A_q y - \mathcal{A}_h P_h y)(t) \, \mathrm{d}t, \quad t \geqslant 0,
$$
where $ e_h(t) := y_h(t) - P_h y(t) $ for all $ t \geqslant 0 $.
By Theorem 3.2 in \cite{Geissert2006}, for any $ \beta \in [0,1] $, it holds almost surely that
$$
\norm{\mathcal{A}_h^{-\beta} e_h'}_{L^p(\mathbb{R}_+; X_{0,h}^q)} + \norm{\mathcal{A}_h \mathcal{A}_h^{-\beta} e_h}_{L^p(\mathbb{R}_+; X_{0,h}^q)} \leqslant C_{p,q,\varrho,\mathcal{O}} \norm{\mathcal{A}_h^{-\beta}(P_h A_q y - \mathcal{A}_h P_h y)}_{L^p(\mathbb{R}_+; X_{0,h}^q)}.
$$
By \cite[Corollary~1.14]{Lunardi2018}, this implies that, almost surely, 
\begin{align*}
&\sup_{t \in \mathbb{R}_+} \norm{\mathcal{A}_h^{-\beta}e_h(t)}_{(X_{0,h}^q,X_{1,h}^q)_{1-1/p,p}}
+ \norm{\mathcal{A}_h^{1-\beta}e_h}_{L^p(\mathbb{R}_+; X_{0,h}^q)} \\
  \leqslant{} & C_{p,q,\varrho,\mathcal{O}} \norm{\mathcal{A}_h^{-\beta}(P_hA_q y - \mathcal{A}_h P_h y)}_{L^p(\mathbb{R}_+; X_{0,h}^q)}.
\end{align*}
For the right-hand side, we estimate as follows:
\[
  \begin{aligned}
& \norm{\mathcal{A}_h^{-\beta}(P_hA_q y - \mathcal{A}_h P_h y)}_{L^p(\mathbb{R}_+; X_{0,h}^q)} \\
    ={}& \norm{\mathcal{A}_h^{1-\beta} (R_h y - P_h y)}_{L^p(\mathbb{R}_+; X_{0,h}^q)}
       && \text{(since $\mathcal{A}_h^{-1} P_h A_q y = R_h y$)} \\
    ={}& C_{\beta,p,q,\varrho,\mathcal{O}} h^{2(\beta-1)} \norm{R_h y - P_h y}_{L^p(\mathbb{R}_+; X_{0,h}^q)}
       && \text{(by \cref{lem:etAh}(ii))}.
  \end{aligned}
\]
Using \cref{lem:Ph}(ii) with $ \alpha = 1/2 $ and \cref{lem:Rh}(ii), we further estimate
$$
\norm{\mathcal{A}_h^{-\beta}(P_h A_q y - \mathcal{A}_h P_h y)}_{L^p(\mathbb{R}_+; X_{0,h}^q)} \leqslant C_{\beta,p,q,\varrho,\mathcal{O}} h^{2\beta - 1} \norm{y}_{L^p(\mathbb{R}_+; X_{1/2}^q)}.
$$
Therefore, for any $ \beta \in [0,1] $, we obtain almost surely that
$$
\sup_{t \in \mathbb{R}_+} \norm{\mathcal{A}_h^{-\beta} e_h(t)}_{(X_{0,h}^q,X_{1,h}^q)_{1-1/p,p}} + \norm{\mathcal{A}_h^{1-\beta} e_h}_{L^p(\mathbb{R}_+; X_{0,h}^q)} \leqslant C_{\beta,p,q,\varrho,\mathcal{O}} h^{2\beta - 1} \norm{y}_{L^p(\mathbb{R}_+; X_{1/2}^q)}.
$$
It follows that, for any $\beta \in [0,1]$,
\begin{equation}
  \label{eq:Phy-yh}
  \begin{aligned}
   & \biggl[
     \mathbb E\sup_{t \in \mathbb R_{+}}
     \norm{\mathcal{A}_h^{-\beta}(y_h - P_hy)(t)}_{(X_{0,h}^q,X_{1,h}^q)_{1-1/p,p}}^p
   \biggr]^{1/p} 
   + \norm{\mathcal{A}_h^{1-\beta}(y_h - P_hy)}_{L^p(\Omega\times\mathbb R_{+};X_{0,h}^q)} \\
    \leqslant{} &
    C_{\beta,p,q,\varrho,\mathcal O}
    h^{2\beta-1} \norm{y}_{L^p(\Omega\times\mathbb R_{+};X_{1/2}^q)}.
  \end{aligned}
\end{equation}

\medskip\textbf{Step 2. Proof of the error estimate \eqref{eq:conv1}.}
By the triangle inequality, we have
\[
\norm{y - y_h}_{L^p(\Omega \times \mathbb{R}_+; X_0^q)}
\leqslant \norm{y - P_h y}_{L^p(\Omega \times \mathbb{R}_+; X_0^q)} + \norm{y_h - P_h y}_{L^p(\Omega \times \mathbb{R}_+; X_0^q)}.
\]
For the first term, \Cref{lem:Ph}(ii) with \( \alpha = 1/2 \) implies
\[
\norm{y - P_h y}_{L^p(\Omega \times \mathbb{R}_+; X_0^q)} \leqslant C_{q,\varrho,\mathcal{O}} \cdot h \norm{y}_{L^p(\Omega \times \mathbb{R}_+; X_{1/2}^q)}.
\]
For the second term, we apply estimate \eqref{eq:Phy-yh} with \( \beta = 1 \), yielding
\[
\norm{y_h - P_h y}_{L^p(\Omega \times \mathbb{R}_+; X_0^q)} \leqslant C_{p,q,\varrho,\mathcal{O}} \cdot h \norm{y}_{L^p(\Omega \times \mathbb{R}_+; X_{1/2}^q)}.
\]
Combining these two estimates gives
\[
\norm{y - y_h}_{L^p(\Omega \times \mathbb{R}_+; X_0^q)} \leqslant C_{p,q,\varrho,\mathcal{O}} \cdot h \norm{y}_{L^p(\Omega \times \mathbb{R}_+; X_{1/2}^q)}.
\]
To conclude the proof of \eqref{eq:conv1}, we invoke the regularity result
\begin{equation}
\label{eq:y-regu}
\begin{aligned}
\left[ \mathbb{E} \sup_{t \geqslant 0} \norm{y(t)}_{(X_0^q, X_1^q)_{1/2 - 1/p, p}}^p \right]^{1/p} + \norm{y}_{L^p(\Omega \times \mathbb{R}_+; X_{1/2}^q)} 
\leqslant C_{p,q,\mathcal{O}} \norm{f}_{L^p(\Omega \times \mathbb{R}_+; \gamma(H, X_0^q))},
\end{aligned}
\end{equation}
which follows from the fundamental work of van Neerven, Veraar, and Weis~\cite[Theorems 1.1 and 1.2]{Neerven2012}.

\medskip\textbf{Step 3 (Proof of Error Estimate \cref{eq:conv2}).}  
To establish the error estimate \cref{eq:conv2} for \( \alpha \in [0,1/p] \), we proceed as follows:
\[
\begin{aligned}
&\biggl[ \mathbb{E} \sup_{t \geqslant 0} \norm{(y_h - P_h y)(t)}_{X_{-\alpha}^q}^p \biggr]^{1/p} \\
  \stackrel{\text{(i)}}{\leqslant}{}& C_{\alpha,q,\varrho,\mathcal{O}} \biggl[ \mathbb{E} \sup_{t \geqslant 0} \norm{(y_h - P_h y)(t)}_{X_{-\alpha,h}^q}^p \biggr]^{1/p} \\
  ={} & C_{\alpha,q,\varrho,\mathcal{O}} \biggl[ \mathbb{E} \sup_{t \geqslant 0} \norm{\mathcal{A}_h^{-(1-1/p+\alpha)}(y_h - P_h y)(t)}_{X_{1-1/p,h}^q}^p \biggr]^{1/p} \\
  \stackrel{\text{(ii)}}{\leqslant}{}& C_{\alpha,p,q,\varrho,\mathcal{O}} (1 + |\ln h|^{1-1/p}) \biggl[ \mathbb{E} \sup_{t \geqslant 0} \norm{\mathcal{A}_h^{-(1-1/p+\alpha)}(y_h - P_h y)(t)}_{(X_{0,h}^q, X_{1,h}^q)_{1-1/p,p}}^p \biggr]^{1/p} \\
  \stackrel{\text{(iii)}}{\leqslant}{} & C_{\alpha,p,q,\varrho,\mathcal{O}} (1 + |\ln h|^{1-1/p}) h^{1+2(\alpha-1/p)} \norm{y}_{L^p(\Omega \times \mathbb{R}_+; X_{1/2}^q)}.
\end{aligned}
\]
Here, step (i) follows from Lemma \ref{lem:lxy}, step (ii) uses Lemma \ref{lem:key}, and step (iii) is derived from \cref{eq:Phy-yh} with \( \beta = 1-1/p+\alpha \).
Additionally, we have
\[
\biggl[ \mathbb{E} \sup_{t \geqslant 0} \norm{(y - P_h y)(t)}_{X_{-\alpha}^q}^p \biggr]^{1/p} \leqslant C_{\alpha,p,q,\varrho} h^{1+2(\alpha-1/p)} \biggl[ \mathbb{E} \sup_{t \geqslant 0} \norm{y(t)}_{(X_0^q, X_1^q)_{1/2-1/p,p}}^p \biggr]^{1/p},
\]
where we apply \cite[Theorem~1.6]{Lunardi2018} along with the standard bounds:
\[
\norm{I - P_h}_{\mathcal{L}(X_0^q, X_{-\alpha}^q)} \leqslant C_{\alpha,q,\varrho} h^{2\alpha}, \quad \norm{I - P_h}_{\mathcal{L}(X_1^q, X_{-\alpha}^q)} \leqslant C_{\alpha,q,\varrho} h^{2+2\alpha}.
\]
Combining these results with the regularity of \( y \) as stated in \cref{eq:y-regu}, we establish the desired pathwise uniform convergence estimate \cref{eq:conv2}. This completes the proof of Theorem \ref{thm:conv}.

\section{Concluding Remarks}
\label{sec:conclusion}
In this work, we have established a key property: the boundedness of the \( H^\infty \)-calculus for the negative discrete Laplacian is independent of the spatial mesh size.
This fundamental result has allowed us to prove a discrete stochastic maximal \( L^p \)-regularity estimate for semidiscrete finite element approximations
of a linear stochastic heat equation. Moreover, we have obtained a nearly optimal estimate for pathwise uniform convergence in the setting of general spatial \( L^q \)-norms.

The numerical analysis of nonlinear stochastic partial differential equations
within the context of general spatial \(L^q\)-norms is a theoretically significant
yet underexplored field. The theoretical advancements presented in this paper will
be useful for future numerical investigations in this area.

\end{document}